\documentclass{amsart}

\usepackage{latexsym, amssymb, amsmath, amscd}
\usepackage{graphics, psfrag}
\usepackage{mathpazo}
\usepackage{color}
\usepackage{fullpage}

\usepackage[all]{xy}

\newcommand{\Ga}{\Gamma}
\newcommand{\Up}{\Upsilon}

\newcommand{\ep}{\epsilon}

\newcommand{\si}{\sigma}

\newcommand{\ga}{\gamma}

\newcommand{\bA}{\mathbb{A}}
\newcommand{\bC}{\mathbb{C}}
\newcommand{\bE}{\mathbb{E}}

\newcommand{\bP}{\mathbb{P}}
\newcommand{\bQ}{\mathbb{Q}}
\newcommand{\bR}{\mathbb{R}}

\newcommand{\bZ}{\mathbb{Z}}

\newcommand{\cB}{\mathcal{B}}
\newcommand{\cC}{\mathcal{C}}

\newcommand{\cE}{\mathcal{E}}
\newcommand{\cF}{\mathcal{F}}
\newcommand{\cH}{\mathcal{H}}
\newcommand{\cI}{\mathcal{I}}
\newcommand{\cK}{\mathcal{K}}
\newcommand{\cIX}{\mathcal{IX}} 
\newcommand{\cL}{\mathcal{L}}
\newcommand{\cM}{\mathcal{M}}

\newcommand{\cO}{\mathcal{O}}

\newcommand{\cQ}{\mathcal{Q}}

\newcommand{\cU}{\mathcal{U}}

\newcommand{\cX}{\mathcal{X}}

\newcommand{\age}{\mathrm{age}}
\newcommand{\Aut}{\mathrm{Aut}}

\newcommand{\ch}{\mathrm{ch}}
\newcommand{\Conj}{\mathrm{Conj}}

\newcommand{\CR}{ {\mathrm{CR}} }
\newcommand{\ev}{\mathrm{ev}}
\newcommand{\Ext}{\mathrm{Ext}}
\newcommand{\Hom}{\mathrm{Hom}}
\newcommand{\id}{\mathrm{id}}
\newcommand{\Ker}{\mathrm{Ker}}

\newcommand{\Ob}{\mathrm{Ob}}
\newcommand{\Out}{\mathrm{Out}}

\newcommand{\Pic}{\mathrm{Pic}}

\newcommand{\rank}{\mathrm{rank}}
\newcommand{\rig}{\mathrm{rig}}
\newcommand{\Spec}{\mathrm{Spec}}

\newcommand{\tw}{\mathrm{tw}}

\newcommand{\vir}{ \mathrm{vir} }
\newcommand{\Eff}{\mathrm{Eff}}

\newcommand{\be}{ {\mathbf{e}} }
\newcommand{\bh}{ {\mathbf{h}} }

\newcommand{\uu}{ {\mathbf{u}} }

\newcommand{\bw}{ {\mathbf{w}} }
\newcommand{\bx}{ {\mathbf{x}} }

\newcommand{\fl}{\mathfrak{l}}

\newcommand{\fo}{\mathfrak{o}}
\newcommand{\fp}{\mathfrak{p}}
\newcommand{\fx}{\mathfrak{x}}
\newcommand{\fy}{\mathfrak{y}}

\newcommand{\vGa}{ {\vec{\Ga }} }

\newcommand{\vUp}{ {\vec{\Up} } }

\newcommand{\vc}{\vec{c}}

\newcommand{\vf}{\vec{f}}
\newcommand{\vg}{\vec{g}}
\newcommand{\vi}{ {\vec{i}} }

\newcommand{\vs}{\vec{s}}

\newcommand{\tT}{ {\widetilde{T}} }

\newcommand{\tcC}{ {\widetilde{\cC}} }

\newcommand{\tf }{\widetilde{f}}

\newcommand{\bIX}{\overline{\cI}\cX}

\newcommand{\hG}{\widehat{G}}
\newcommand{\hbeta}{\hat{\beta}}
\newcommand{\hga}{\hat{\gamma}}

\newcommand{\Mbar}{\overline{\cM}}

\newcommand{\MgX}{\Mbar_{g,n}(X,\beta)}

\newcommand{\MgcX}{\Mbar_{g,n}(\cX,\beta)}
\newcommand{\MgXi}{\Mbar_{g,\vi}(\cX,\beta)}
\newcommand{\MBG}{\Mbar_{g,n}(\cB G)}
\newcommand{\MBGc}{\Mbar_{g,\vc}(\cB G)}

\newcommand{\MghX}{\Mbar_{g,n}(\hXY,\hbeta)}
\newcommand{\MghXi}{\Mbar_{g,\vi}(\hXY,\hbeta)}

\newcommand{\hXY}{\hat{\cX}_{\vUp}}
\newcommand{\IY}{ { I_{\vUp } }}

\newcommand{\edge}{ {\tiny\begin{array}{c}e\in E(\Ga) \\(e,v),(e,v')\in F(\Ga)\end{array}} }

\newcommand{\one}{\mathbb{1}}
\newcommand{\pt}{\textup{point}}

\newcommand{\lra}{\longrightarrow}
\newcommand{\hra}{\hookrightarrow}

\numberwithin{equation}{section}

\newtheorem{dummy}{dummy}[section]
\newtheorem{theorem}[dummy]{Theorem}
\newtheorem{proposition}[dummy]{Proposition}

\newtheorem{example}[dummy]{Example}
\newtheorem{definition}[dummy]{Definition}
\newtheorem{remark}[dummy]{Remark}

\newtheorem{lemma}[dummy]{Lemma}
\newtheorem{assumption}[dummy]{Assumption}

\begin{document}
\setcounter{page}{1}
%
%
\long\def\replace#1{#1}

\title{Stacky GKM Graphs and Orbifold Gromov-Witten Theory}
    
\author{Chiu-Chu Melissa Liu}
\address{Chiu-Chu Melissa Liu, Department of Mathematics, Columbia University, New York, NY 10027, USA}
\email{ccliu@math.columbia.edu}

\author{Artan Sheshmani}
\address{Artan Sheshmani, Center for Mathematical Sciences and Applications, Harvard University, Department of Mathematics, 20 Garden Street, Room 207, Cambridge, MA, 02139}
\address{Centre for Quantum Geometry of Moduli Spaces, Aarhus University, Department of Mathematics, Ny Munkegade 118, building 1530, 319, 8000 Aarhus C, Denmark}
\address{National Research University Higher School of Economics, Russian Federation, Laboratory of Mirror Symmetry, NRU HSE, 6 Usacheva str., Moscow, Russia, 119048}
\email{artan@cmsa.fas.harvard.edu}


\begin{abstract}
A smooth GKM stack is a smooth Deligne-Mumford stack equipped with an action of 
an algebraic torus $T$, with only finitely many zero-dimensional and one-dimensional orbits. 
\begin{enumerate}
\item[(i)]  We define the stacky GKM graph of a smooth GKM stack, under the mild assumption that any one-dimensional $T$-orbit
closure contains at least one $T$ fixed point. The stacky GKM graph is a decorated graph which contains enough information
to reconstruct the $T$-equivariant formal neighborhood of the 1-skeleton  (union of zero-dimensional and one-dimensional 
$T$-orbits) as a formal smooth DM stack equipped with a $T$-action. 
\item[(ii)] We axiomize the definition of a stacky GKM 
graph and introduce abstract stacky GKM graphs which are more general than stacky GKM graphs of honest smooth GKM stacks. From an 
abstract GKM graph we construct a formal smooth GKM stack. 
\item[(iii)] We define equivariant orbifold Gromov-Witten invariants of smooth GKM stacks, as well as formal equivariant orbifold Gromov-Witten invariants 
of formal smooth GKM stacks. These invariants can be computed by virtual localization and depend 
only the stacky GKM graph or the abstract stacky GKM graph. Formal equivariant orbifold Gromov-Witten invariants of the stacky GKM graph of 
a smooth GKM stack $\cX$ are refinements of equivariant orbifold Gromov-Witten invariants of $\cX$.
\end{enumerate}
\end{abstract}

\maketitle

\setcounter{tocdepth}{1} 

\tableofcontents

\section{Introduction}


\subsection{Background and motivation}

\subsubsection{GKM manifolds and GKM graphs}
An algebraic GKM manifold, named after Goresky, Kottwitz, and MacPherson, is a non-singular 
(complex) algebraic variety  equipped  with an action  of an algebraic torus  $T=(\bC^*)^m$ such that there are finitely many zero-dimensional and 
one-dimensional orbits. The action $T$ on $X$ is equivariantly formal over a field $\mathbb{K}$ if
the map from the $T$-equivariant cohomology $H^*_T(X;\mathbb{K})$ of $X$ to the ordinary cohomology $H^*(X;\mathbb{K})$ of $X$ is surjective.
In \cite{GKM}, Goresky-Kottwitz-MacPherson showed that  (when $\mathbb{K}=\bR$) the 
$T$-equivariant cohomology and the ordinary cohomology of an equivariantly formal GKM manifold can be expressed in terms  of 
a decorated graph known as the GKM graph. As an abstract graph, the vertices and edges of the GKM graph are in one-to-one correspondence with the zero-dimensional and one-dimensional orbits of the $T$-action on the GKM manifold. The additional decorations  provide enough information to reconstruct the $T$-equivariant formal neighborhood of the 1-skeleton (the union  of zero-dimensional and one-dimensional orbits of the torus action) of the GKM manifold. Toric manifolds are examples of GKM manifolds.

\subsubsection{Equivariant Gromov-Witten theory of GKM manifolds}
The quantum cohomology of a projective manifold $\cX$ is equal to the rational cohomology as a vector space over $\bQ$, equipped with the quantum product which is a family of products parametrized by Novikov variables such that the classical cup product is recovered by setting all the Novikov variables to zero. The quantum product is determined by genus zero Gromov-Witten invariants which are virtual counts of rational curves in $\cX$. More generally, genus $g$ Gromov-Witten invariants are virtual counts of genus $g$ stable maps to $\cX$. When $\cX$ is a GKM manifold which is equivariantly formal over $\bQ$,
Gromov-Witten invariants (which are rational numbers) can be lifted to equivariant Gromov-Witten invariants which take values in $H^*(\cB T;\bQ)\cong \bQ[u_1,\ldots, u_m]$, where $\cB T$ is the classifying space of the torus $T$.  Equivariant Gromov-Witten invariants can be computed by virtual localization on moduli of stable maps to $\cX$ \cite{Be2, GP98, GrPa} and depend only on the GKM graph \cite{LS}; the formula also makes sense for non-compact GKM manifolds if one works over the fractional field 
$\bQ(u_1,\ldots,u_m)$ of $H^*(\cB T;\bQ)=\bQ[u_1,\ldots, u_m]$. In particular, equivariant and non-equivariant quantum cohomology of an equivariantly formal projective GKM manifold is an invariant of the GKM graph (and the semi-group of effective curve classes which determine the Novikov ring).

\subsubsection{Generalization to orbifolds}
 The GKM theory and GKM graphs have been generalized to orbifolds. To our knowledge, GKM graphs have been defined for smooth 
orbifolds having presentation as a global quotient of a smooth manifold by action of a torus \cite{GZ2}. 
In the present paper we consider the most general possible definition that we can think of in the algebraic setting: 
we define a smooth GKM stack to be a smooth Deligne-Mumford (DM) stack equipped with an action by an algebraic torus, 
with finitely many zero-dimensional and one-dimensional orbits. In \cite{Zo}, Givental's quantization formula for all genus full descendant potential
of equivariant Gromov-Witten (GW) theory of GKM manifolds is generalized to equivariant orbifold GW theory of smooth GKM stacks; the definition of a smooth GKM stack in this paper is the same as the definition of a GKM orbifold in \cite{Zo}.
One of the main goals of this paper is to provide details of localization computations of orbifold GW invariants 
in the generality needed in \cite{Zo} which is beyond the toric stack case in \cite{Liu}.

\subsubsection{Examples from geometric engineering and mirror symmetry}
In \cite{LLLZ}, Li-Liu-Liu-Zhou introduced formal toric Calabi-Yau graphs and formal toric Calabi-Yau threefolds, which
arise in geometric engineering in 5d and 6d (see e.g. Figure 1, 20-22, 25, 26, 29-32 of \cite{HIV}) 
and Bryan's recent work on Donaldson-Thomas invariants of the banana manifold (Figure 2 of \cite{B19}). Formal toric Calabi-Yau threefolds
are usually not equivariantly formal, but their Gromov-Witten invariants and Donaldson-Thomas invariants can be defined and computed by the topological vertex. In this paper, we introduce abstract stacky GKM graphs and formal smooth GKM stacks, which 
include (as a very special case) the orbifold version of formal toric Calabi-Yau graphs and formal toric Calabi-Yau threefolds, 
of which orbifold Donaldson-Thomas invariants can be defined and computed by the orbifold topological vertex \cite{BCY}. 

Formal toric Calabi-Yau manifolds of arbitrary dimension arise in Strominger-Yau-Zaslow (SYZ) mirror symmetry for varieties of general type \cite{KL}.
More precisely, such an example is obtained by taking a formal neighborhood of the 1-skeleton 
of a non-toric non-compact K\"{a}hler manifold equipped with a compact torus action which is not Hamiltonian. 
These examples  naturally extend to orbifolds, which are special cases of formal smooth GKM stacks defined in this paper.  

\subsection{Summary of results}
In this paper, algebraic varieties and algebraic stacks are defined over the  field $\bC$ of complex numbers. A brief review
of smooth DM stacks is given in Section \ref{sec:DMstacks}.
\subsubsection{Smooth GKM stacks and stacky GKM graphs}
In Section \ref{sec:GKM}, we define the stacky GKM graph of a smooth GKM stack, under the mild assumption that any one-dimensional orbit contains at least one torus fixed point.  As an abstract graph, the vertices and edges are in one-to-one correspondence with the zero-dimensional and one-dimensional orbits of the torus action, respectively. 
\begin{enumerate}
\item In the manifold case, a zero-dimensional orbit is a (scheme) point, 
and a one-dimensional orbit closure is isomorphic to the complex projective line $\mathbb{P}^1$ or a complex affine line $\bC$.
\item In the case of global quotients by torus action (such as smooth toric DM stacks), a zero-dimensional orbit is 
a stacky point $\cB G =[\mathrm{point}/G]$ where $G$ is a finite abelian group, and
a one-dimensional orbit closure is a one-dimensional smooth toric DM stack.  
\item The general case studied in the present paper is significantly more complicated and subtle than the special case (2) studied in previous work:
in the general case, a zero-dimensional orbit is of the form $\cB G =[\mathrm{point}/G]$ where
G is a possibly non-abelian finite group, and a proper one-dimensional orbit closure is a spherical DM curve in the sense of Behrend-Noohi 
\cite{BN}.  
\end{enumerate}
Thanks to Behrend-Noohi's presentation of a general spherical DM curve $\fl$ as a quotient stack of 
the form  $[(\bC^2-\{0\})/E]$, where $E$ is a central extension of the fundamental group  $\pi_1(\fl)$ of the DM curve $\fl$
by $\bC^*$ \cite{BN}, we may express the total space of the normal bundle $N_{\fl/\cX}$ of a proper
one-dimensional orbit closure $\fl$ in a general smooth GKM stack $\cX$ as a global quotient $[((\bC^2-\{0\})\times \bC^{r-1})/E]$, where $r$ is the 
dimension of $\cX$, though $\cX$ itself is not necessary a global quotient. 
Such a presentation is crucial for the following two steps:  (i) to identify the extra decorations that we need to include in the definition of 
a stacky GKM graph, so that the equivariant formal neighborhood of the 1-skeleton can be reconstructed; (ii) to describe the torus 
fixed locus in the moduli stack of twisted stable maps to $\cX$, which is the first step of localization computations of equivariant
orbifold GW invariants of $\cX$. 

Smooth toric DM stacks are examples of smooth GKM stacks; the stacky GKM graph of a smooth toric DM stack  is determined by 
the stacky fan defining the smooth toric DM stack. See \cite{BCS} for definitions of stacky fans and smooth toric DM stacks.

\subsubsection{Abstract stacky GKM graphs and formal smooth GKM stacks}
In Section \ref{sec:abstract-formal},  we axiomize the definition of a stacky GKM graph and 
introduce abstract stacky GKM graphs which are more general than stacky GKM graphs of honest smooth GKM stacks. 
From an abstract stacky GKM graph we construct a formal smooth GKM stack.  Our definition
of abstract stacky GKM graphs is inspired by the definition of abstract 1-skeleton in \cite{GZ2} and also generalizes 
it in several aspects. 

Given a formal smooth GKM stack $\hXY$ defined by an abstract GKM graph $\vUp$, we define
the set of effective classes $\Eff(\hXY)$ and a vector space $\cH_{\vUp}$ over the fractional field
$\cQ_T=\bQ(u_1,\ldots,u_m)$ of $R_T:= H^*(\cB T;\bQ) =\bQ[u_1,\ldots,u_m]$. If $\hXY$ is the stacky GKM graph of 
an honest smooth GKM stack $\cX$, then there is a surjective map $j_*: \Eff(\hXY) \to \Eff(\cX)$ and 
a $\cQ_T$ linear map $j^*:H_T^*(\hXY;\cQ_T)\to \cH_{\vUp}$. If the $T$-action on $\cX$ is equivariant formal then  $j^*$ is a linear isomorphism.

\subsubsection{Equivariant orbifold Gromov-Witten theory} 
In Section \ref{sec:orbGW}, we define equivariant orbifold Gromov-Witten (GW) invariants
of smooth GKM stacks and formal equivariant orbifold GW invariants of formal smooth GKM stacks.
Let $g$, $n$, $a_1,\ldots,a_n$ be non-negative integers. 
\begin{enumerate}
\item[(i)] Let $\cX$ be a smooth GKM stack and let $T$ be the complex algebraic torus acting
on $\cX$. Given an effective curve class $\beta\in \Eff(\cX)$
and $T$-equivariant cohomology classes $\gamma_1^T,\ldots,\gamma_n^T\in H^*_T(\cX;\cQ_T)$, where
$\cQ_T =\bQ(u_1,\ldots, u_m)$ is the fractional field of $H^*(\cB T)=\bQ[u_1,\ldots,u_m]$, 
we define genus $g$, degree $\beta$, $T$-equivariant orbifold GW invariants
$$
\langle \bar{\ep}_{a_1}(\gamma_1^T),\ldots, \bar{\ep}_{a_n}(\gamma_n^T)\rangle ^{\cX_T}_{g,\beta}
\in \cQ_T
$$
via localization. When the coarse moduli space of $\cX$ is projective (so that non-equivariant orbifold Gromov-Witten invariants
of $\cX$ are defined) and the torus action on $\cX$ is equivariantly formal over $\bQ$ (in the sense that the map from
$T$-equivariant orbifold Chen-Ruan cohomology to the non-equivariant orbifold Chen-Ruan cohomology is surjective), they 
are refinements of (non-equivariant) orbifold GW invariants of $\cX$.

\item[(ii)] Given a formal smooth GKM stack $\hXY$ defined by an abstract GKM graph $\vUp$, we define
the set of effective classes $\Eff(\hXY)$ and a vector space $\cH_{\vUp}$ over $\cQ_T$.
Given $\hbeta\in \Eff(\hXY)$ and $\hga_1,\ldots,\hga_n\in \cH_{\vUp}$, 
we define genus $g$, degree $\hbeta$, formal $T$-equivariant orbifold GW invariants
\begin{equation}\label{eqn:formal-invariants}
\langle \bar{\ep}_{a_1}(\hga_1),\ldots, \bar{\ep}_{a_n}(\hga_n)\rangle ^{\vGa}_{g,\hbeta}.
\in \cQ_T
\end{equation}
via localization.

\item[(iii)] If $\vGa$ is the stacky GKM graph of $\cX$ in (i) then there is 
a surjective map $j_*: \Eff(\hXY)\lra \Eff(\cX)$ and a $\cQ_T$-linear map
$j^*: H_T^*(\cX;\cQ_T) \lra \cH_\vUp$ such that
$$
\langle \bar{\ep}_{a_1}(\gamma_1^T),\ldots, \bar{\ep}_{a_n}(\gamma_n^T)\rangle ^{\cX_T}_{g,\beta}
=\sum_{\substack{ \hbeta\in \Eff(\hXY)\\ j_*\hbeta=\beta } }  \bar{\ep}_{a_1}(j^*\gamma_1^T),\ldots, \bar{\ep}_n(j^*\gamma_n^T)\rangle ^{\vGa}_{g,\hbeta}
$$
Therefore, formal equivariant orbifold GW invariants of $\vGa$ are refinements of equivariant orbifold GW invariants
of the smooth GKM stack $\cX$.
\end{enumerate}

In Section \ref{sec:localization},
we derive explicit localization formula of equivariant orbifold GW invariants \eqref{eqn:formal-invariants}. 
The localization formula is stated as Theorem \ref{main-orb}. In particular, we obtain localization formula of equivariant
orbifold GW invariants of smooth GKM stacks.

\subsection*{Acknowledgments}
The first author wishes to thank Tom Graber for his suggestion of generalizing
the localization computations for smooth toric DM stacks in \cite{Liu} to smooth GKM stacks, and Johan de Jong and Daniel Litt for helpful 
conversations. The second author would like to sincerely thank Behrang Noohi for teaching him his joint work with Kai Behrend \cite{BN} on DM curves. Without Noohi's generous instructions, which assisted the authors to formulate the abstract stacky GKM graphs, this work could not be completed. The first author would also like to sincerely thank Kai Behrend for very helpful discussion on geometry of DM curves and results in \cite{BN}. 

The first author is partially supported by NSF DMS-1206667 and NSF DMS-1564497; she was also supported by NSF DMS-1440140 while she was in residence at the Mathematical Sciences Research Institute (MSRI) in Berkeley, California, during Spring 2018.
The second author was partially supported by NSF DMS-1607871, NSF DMS-1306313 and Laboratory of Mirror Symmetry NRU HSE, RF Government grant, ag. No 14.641.31.0001. The second author would like to further sincerely thank the center for Quantum Geometry of Moduli Spaces at Aarhus University, the Center for Mathematical Sciences and Applications at Harvard University and the Laboratory of Mirror Symmetry in Higher School of Economics, Russian federation, for the great help and support.

\section{Smooth Deligne-Mumford Stacks} \label{sec:DMstacks}
Let $\cX$ be a smooth Deligne-Mumford (DM) stack, and let $\pi:\cX\to X$ be the natural projection to the coarse moduli space $X$.

\subsection{The inertia stack and its rigidification}
The inertia stack $\cIX$ associated to $\cX$ is a smooth DM stack
such that the following diagram is Cartesian:
$$
\begin{CD}
\cIX @>>> \cX \\
@VVV @VV{\Delta}V \\
\cX  @>{\Delta}>> \cX\times \cX 
\end{CD}
$$
where $\Delta:\cX\to \cX\times \cX$ is the diagonal map.
An object in the category $\cIX$ is a pair
$(x,g)$, where $x$ is an object in the category
$\cX$ and $g\in \Aut_{\cX}(x)$:
$$
\Ob(\cIX) =\{ (x,g)\mid x\in \Ob(\cX), g\in \Aut_{\cX}(x)\}.
$$
The morphisms between two objects in the category $\cIX$ are:
$$
\Hom_{\cIX}( (x_1,g_1), (x_2,g_2)) =\{ h\in \Hom_{\cX} (x_1, x_2) \mid h\circ g_1 = g_2 \circ h \}.
$$
In particular,
$$
\Aut_{\cIX}(x,g) =\{ h\in \Aut_{\cX}(x)\mid h\circ g = g\circ h\}.
$$
The rigidified inertia stack $\bIX$ satisfies
$$
\Ob(\bIX) = \Ob(\cIX),\quad
\Aut_{\bIX}(x,g) = \Aut_{\cIX}(x,g)/\langle g\rangle,
$$
where $\langle g\rangle$ is the subgroup of  $\Aut_{\cIX}(x,g)$ generated by $g$.

There is a natural projection $q:\cIX\to \cX$ 
which sends $(x,g)$ to $x$.  There is a natural
involution $\iota:\cIX\to \cIX$ which
sends $(x,g)$ to $(x,g^{-1})$. We assume that $\cX$ is connected. Let
$$
\cIX =\bigsqcup_{i\in I}\cX_i
$$
be disjoint union of connected components.
There is a distinguished connected component
$\cX_0$ whose objects are $(x,\id_x)$, where
$x\in \Ob(\cX)$, and $\id_x\in \Aut(x)$ is the identity
element; note that $\cX_0 \cong \cX$. The involution $\iota$ restricts
to an isomorphism $\iota_i:\cX_i\to \cX_{\iota(i)}$. 
In particular, $\iota_0:\cX_0\to \cX_0$ is the
identity functor.

\begin{example}[classifying space] \label{ex:BG}
Let $G$ be a finite group. The stack $\cB G= [\pt/G]$
is a category which consists of one object $x$, and $\Hom(x,x)=G$. 
The objects of its inertia stack $\cI \cB G$ are
$$
\Ob(\cI \cB G) =\{ (x,g)\mid g\in G\}.
$$
The morphisms between two objects are
$$
\Hom((x,g_1), (x,g_2)) = \{ g\in G\mid g_2 g = g g_1\} =\{ g\in G\mid g_2= g g_1 g^{-1}\}. 
$$
Therefore
$$
\cI\cB G\cong [G/G]
$$
where $G$ acts on $G$ by conjugation. We have
$$
\cI \cB G=\bigsqcup_{c\in \Conj(G)} (\cB G)_c
$$
where $\Conj(G)$ is the set of conjugacy classes in $G$, and
$(\cB G)_c$ is the connected component associated to the conjugacy 
class $c\in \Conj(G)$. We have
$$
(\cB G)_c = [c/G] \cong [\{h\}/ C_G(h)] \cong \cB \left(C_G(h)\right).
$$
for any element $h$ in the conjugacy class $c$, where $C_G(h)=\{ a\in G: ah=ha\}$ is the centralizer of $h$
in the group $G$.

In particular, when $G$ is abelian, we have $\Conj(G)=G$, and
$$
\cI \cB G =\bigsqcup_{h\in G} (\cB G)_h
$$
where $(\cB G)_h =[\{ h \} /G]\cong \cB G$.

\end{example}

\subsection{Age} \label{sec:age}
Given a positive integer $s$, let $\mu_s$ denote the 
group of $s$-th roots of unity. It is a cyclic subgroup
of $\bC^*$ of order $s$, generated by
$$
\zeta_s := e^{2\pi \sqrt{-1}/s}.
$$

Given any object $(x,g)$ in $\cIX$, $g:T_x\cX\to T_x\cX$ is 
a linear isomorphism such that $g^s=\id$, where $s$ is the order
of $g$. The eigenvalues of $g:T_x\cX\to T_x\cX$ are $\zeta_s^{l_1},\ldots, \zeta_s^{l_r}$, where
$l_i\in \{0,1,\ldots, s-1\}$ and $r=\dim_\bC \cX$.
Define 
$$
\age(x,g):= \frac{l_1+\cdots+l_r}{s}. 
$$
Then $\age: \cIX\to \bQ$ is constant on each connected
component $\cX_i$ of $\cIX$.  Define $\age(\cX_i)=\age(x,g)$ where
$(x,g)$ is any object in the groupoid $\cX_i$.  Note that
$$
\age(\cX_i)+ \age(\cX_{\iota(i)})=\dim_\bC \cX-\dim_\bC \cX_i.
$$

\subsection{The Chen-Ruan orbifold cohomology group}
In \cite{CR1}, W. Chen and Y. Ruan introduced  the orbifold cohomology group 
of a complex orbifold. See \cite[Section 4.4]{AGV1} for a more algebraic version.

As a graded $\bQ$ vector space, the Chen-Ruan orbifold cohomology group of $\cX$ is defined to be
$$
H^*_\CR(\cX;\bQ): = \bigoplus_{a\in \bQ_{\geq 0}}H^a_\CR(\cX;\bQ)
$$
where
$$
H^a_\CR(\cX;\bQ)=\bigoplus_{i\in I} H^{a-2\age(\cX_i)}(\cX_i;\bQ).
$$

Suppose that $\cX$ is proper. Then we have the following proper pushforward to a point:
$$
\int_\cX:H^*(\cX;\bQ)\to H^*(\pt;\bQ) =\bQ.
$$
The orbifold Poincar\'{e} pairing is defined by 
$$
(\alpha,\beta) := 
\begin{cases} 
\int_{\cX_i} \alpha \cup \iota_i^* \beta, & j=\iota(i),\\
0, & j\neq \iota(i),
\end{cases}
$$
where $\alpha\in H^*(\cX_i;\bQ)$, $\beta\in H^*(\cX_j;\bQ)$.

\section{Smooth GKM stacks and stacky GKM graphs} \label{sec:GKM}

In this section, we describe the geometry of
smooth GKM stacks, and define the stacky GKM graph of a smooth GKM stack.
In the algebraic setting, smooth GKM stacks are more general than the GKM orbifolds
in Guillemin-Zara \cite{GZ1, GZ2} in the following ways: 
\begin{enumerate}
\item Guillemin-Zara consider compact GKM manifolds or orbifolds, whereas we consider smooth GKM stacks which are not necessarily compact.
\item By orbifolds Guillemin-Zara mean orbifolds having a presentation of the form $X/K$, $K$ being a torus and $X$ being a manifold on which $K$ acts in a faithful, locally free fashion \cite[Section 1.2]{GZ2}.  In particular, the inertia group  of a point is 
a finite abelian group, and the generic inertia group is trivial.  Our smooth GKM stacks do not have such a presentation in general; the inertia group of a point
is a possibly non-abelian finite group, and the generic inertial group is not necessarily trivial.
\item We do not assume the torus action on the smooth GKM stack is faithful.
\end{enumerate}
On the other hand, Guillemin-Zara work in the $C^\infty$ category and consider $C^\infty$-action
by a compact torus $U(1)^m$, while we restrict ourselves to smooth DM stacks and algebraic
action by an algebraic torus $(\bC^*)^m$ (which restricts to a $U(1)^m$-action).  

\subsection{Smooth GKM stacks}
The following definition of a smooth GKM stack is the same as the definition of a GKM orbifold in \cite{Zo}.

\begin{definition}[smooth GKM stacks]
Let $\cX$ be a smooth DM stack. 
We say $\cX$ is a smooth GKM stack if it is equipped with an action of an algebraic torus
$T=(\bC^*)^m$ with only finitely many zero-dimensional and one-dimensional orbits.
\end{definition}

The notion of a group action on a stack is discussed in \cite{Ro}.

Let $N=\Hom(\bC^*, T) \cong \bZ^m$ be the lattice of 1-parameter
subgroups of $T$, and let $M=\Hom(T,\bC^*)$ be the character lattice of $T$. Then $M=\Hom(N,\bZ)$ is the dual lattice of $N$. We introduce
$$
N_\bR:= N\otimes_{\bZ} \bR,\quad
M_\bR:= M\otimes_{\bZ} \bR,\quad
N_\bQ := N\otimes_\bZ \bQ, \quad M_\bQ := M\otimes_{\bZ}\bQ.
$$
Then $M_\bQ$ can be canonically identified with $H^2_T(\pt;\bQ) =
H^2(\cB T;\bQ)$, where $\cB T$ is the classifying space of $T$.

We make the following assumption on $\cX$.
\begin{assumption}\label{assume}
\begin{enumerate}
\item The set $\cX^T$ of $T$ fixed points in $\cX$ is non-empty.
\item The coarse moduli space of a one-dimensional orbit closure is either a complex projective line $\bP^1$ 
or a complex affine line $\bC$.
\end{enumerate}
\end{assumption}
Note that (1) and (2) hold when $\cX$ is proper. Indeed, if $\cX$ is 
proper then the coarse moduli space of any one-dimensional orbit closure
is $\bP^1$.

\begin{example}
If $\cX$ is a smooth toric DM stack defined by a finite fan \cite{BCS, FMN}, then $\cX$ is a smooth GKM stack.
In particular, any proper smooth toric DM stack is a smooth GKM stack.
\end{example}

\begin{example}[footballs] \label{ex:football}
Given any positive integers $m,n$, define a subgroup $G_{m,n}$ of $(\bC^*)^2$ by 
$$
G_{m,n}=\{ (t_1,t_2)\in (\bC^*)^2: t_1^n=t_2^m\}.
$$
The football $\cF(m,n)$ is defined as the quotient stack
$$
\cF(m,n):= [(\bC^2-\{0\})/G_{m,n}]
$$
where $(t_1,t_2)\in G_{m,n}$ acts on $(z_1,z_2)\in \bC^2-\{0\}$ by $(t_1,t_2)\cdot (z_1,z_2)=(t_1z_1, t_2z_2)$. Then
$\cF(m,n)$ is a proper smooth toric DM stack, so it is a smooth GKM stack.
The inertial groups of the two torus fixed points $[1,0]$ and $[0,1]$ are $\mu_m$ and $\mu_n$, respectively; the inertial group of any other point is trivial. 
\end{example}

\begin{example}[weighted projective lines]\label{ex:WP}
Given any positive integers $m,n$,  the weighted projective line $\bP(m,n)$ is defined as the quotient stack
$$
\bP(m,n):= [ (\bC^2-\{0\})/\bC^*]
$$
where $t\in \bC^*$ acts on $(z_1,z_2)\in \bC^2-\{0\}$ by $t\cdot (z_1,z_2)= (t^m z_1, t^n z_2)$. Then $\bP(m,n)$ is a proper smooth toric DM stack, so
it is a smooth GKM stack. The inertial group of the two torus fixed points
$[1,0]$ and $[0,1]$ are $\mu_m$ and $\mu_n$, respectively; the inertia group of any other point is $\mu_d$, where $d=g.c.d.(m,n)$ is the greatest common divisor of $m,n$.
The rigidification of $\bP(m,n)$ is $\bP(\frac{m}{d},\frac{m}{d})\cong \cF(\frac{m}{d},\frac{n}{d})$.  
\end{example}

\begin{example}
An algebraic GKM manifold (in the sense of \cite{LS}) is a smooth GKM stack.
\end{example}

\begin{definition}
Let $\cX$ be a smooth GKM stack. The {\em 0-skeleton} of $\cX$ is defined to be
$\cX^0:= \cX^T$ which is the union of zero-dimensional orbits of the $T$-action on $\cX$. 
The {\em 1-skeleton} $\cX^1$ of  $\cX$ is defined to be the 
union of zero-dimensional  and one-dimensional orbits of the $T$-action on $\cX$. 
\end{definition}

\subsection{A zero dimensional orbit and its normal bundle} \label{sec:zero-dim} 
Let $\cX$ be an $r$-dimensional smooth GKM stack, so that $T=(\bC^*)^m$ acts algebraically on $\cX$. 
A zero-dimensional $T$ orbit in $\cX$ is a fixed (possibly stacky) point $\fp =\cB G$ under the $T$-action on $\cX$, where $G$ is a finite group. The normal bundle of $\fp$ in $\cX$ is the tangent space $T_{\fp}\cX$ to $\cX$ at $\fp$, which is a rank $r$ vector bundle over $\cB G$, or equivalently, 
a representation $\phi: G\to GL(r,\bC)$. The $T$-action on $\cX$ induces a $T$-action on $T_{\fp}\cX =[\bC^r/G]$, which can be viewed as a $T$-equivariant vector bundle of rank $r$ over $\cB G$. The GKM assumption implies that
$T_{\fp}\cX$ is a direct sum of $T$-equivariant line bundles 
$L_1,\ldots, L_r$ over $\cB G$, so that 
$$
\phi=\bigoplus_{i=1}^r\phi_i 
$$
is the direct sum of $r$ one-dimensional representations $\phi_i: G\to GL(1,\bC)=\bC^*$.
We may choose coordinates on $\bC^r$ such that $L_i$ corresponds to the $i$-th coordinate axis. The $G$-action on $\bC^r$ is given by
$$
g\cdot (z_1,\ldots,z_r)= (\phi_1(g)z_1,\ldots, \phi_r(g)z_r)
$$
where $g\in G$ and $(z_1,\ldots, z_r)\in \bC^r$. Let
$$
\bw_i:= c_1^T(L_i) \in H^2_T(\cB G;\bQ) \cong H^2_T(\pt;\bQ)= M_\bQ. 
$$
The GKM condition implies that $\bw_i$ and $\bw_j$ are linearly independent if $i\neq j$. 
The tangent space $T_{\fp}\cX=[\bC^r/G]$, together with the $T$-action, is an affine smooth
GKM stack characterized by the finite group $G$, $\phi_1,\ldots,\phi_r \in \Hom(G,\bC^*)$, 
and the weights $\bw_1,\ldots, \bw_r\in M_\bQ$. 

We define the stacky GKM graph of $[\bC^r/G]$ as follows. The underlying graph has a single vertex $\si$ and $r$ rays $\ep_1,\ldots, \ep_r$ emanating from 
$\si$. The vertex is decorated by the group $G$; the edge $\ep_i$ is decorated by the group $G_i$;
the flag $(\ep_i, \si)$ is decorated by $\phi_i$, $\bw_i$, and the injective group homomorphism 
$G_i\hookrightarrow G$.

The image of $\phi_i:G\to \bC^*$ is a finite cyclic group $\mu_{r_i}$ of order $r_i>0$.
Let $G_i$ be the kernel of $\phi_i$. For each $i$, we have a short exact sequence of finite groups:
$$
1\to G_i \longrightarrow G \stackrel{\phi_i}{\longrightarrow} \mu_{r_i}\to 1.
$$

The coarse moduli 
$$
\bC^r/G = \Spec \left(\bC[z_1,\ldots,z_r]^G\right)
$$
is an affine $T$ scheme. Let $x_i:= z_i^{r_i}$. The $i$-th axis
$$
\ell_i = [\{ (z_1,\ldots, z_r)\in \bC^r: z_j=0  \textup{ for }j\neq i\}/G] = \Spec \left(\bC[z_i]^G\right) = \Spec \bC[x_i]
$$
is a 1-dimension affine $T$ subscheme of $\bC^r/G$. The $T$-action on $\bC^r$ restricts to a $T$-action 
on $\ell_i\cong \bC$ with weight $r_i \bw_i$, so 
$$
r_i \bw_i \in M. 
$$

\subsection{A proper one-dimensional orbit closure} \label{sec:one-dim}

The main reference of this subsection is \cite{BN}. We thank Behrend 
and Noohi for explaining results in their paper \cite{BN} to us.

Let $\fl \subset \cX$ be a proper one-dimensional $T$ orbit closure in $\cX$. Then $\fl$ contains exactly two zero-dimensional
$T$ orbits $x$ and $y$ with inertia groups $G_x$ and $G_y$, respectively. The representation of $G_x$ (resp.$G_y$) on the tangent
line $T_x \fl$ (resp. $T_y \fl$) determines a group homomorphism $\phi_x: G_x \to  \bC^*$ (resp. $\phi_y: G_y \to \bC^*$) 
with image $\mu_{r_x}$ (resp. $\mu_{r_y}$), where
$r_x$ and $r_y$ are positive integers. Then $\fl$ is a $G$-gerbe over its rigidification  $\fl^{\rig}$, where
$G \cong \Ker(\phi_x)\cong \Ker(\phi_y)$ and $\fl^{\rig}$ is an orbifold DM curve
isomorphic to the football $\cF(r_x,r_y)$ (cf. Example \ref{ex:football}); here an orbifold DM curve is a 1-dimensional smooth DM stack with a trivial generic inertia group.
Let $\bar{x}$ and $\bar{y}$ be the images of $x$ and $y$ under the morphism 
$\fl\to \fl^\rig$. The inertia groups of $\bar{x}$ and $\bar{y}$ are $\mu_{r_x}$ and $\mu_{r_y}$, respectively. The coarse moduli space
$\ell$ of $\fl$ and $\fl^\rig$ is isomorphic to the projective line $\bP^1$. 

The DM curve $\fl$ is spherical in the sense of \cite{BN}. In the rest of this subsection, we recall some relevant facts from \cite{BN}.
\begin{enumerate}
\item The open embeddings
$$
\iota_{\bar{x}}: \cU_{\bar{x}}:=\fl^{\rig} \setminus \{\bar{y}\} \cong [\bC/\mu_{r_x}] \hookrightarrow \fl^{\rig},\quad
\iota_{\bar{y}}: \cU_{\bar{y}}:=\fl^{\rig} \setminus \{\bar{x}\} \cong [\bC/\mu_{r_y}]\hookrightarrow \fl^{\rig} 
$$
induce surjective group homomorphisms
$$
\iota_{\bar{x}*} : \pi_1(\cU_{\bar{x}}) \cong \mu_{r_x} \to \pi_1(\fl^\rig)\cong \mu_a,\quad
\iota_{\bar{y}*} : \pi_1(\cU_{\bar{y}}) \cong \mu_{r_y} \to  \pi_1(\fl^\rig)\cong \mu_a,
$$
where $a=g.c.d.(r_x,r_y)$.
\item The open embeddings
$$
\iota_x: \cU_x:=\fl\setminus \{y\} \cong [\bC/G_x] \hookrightarrow \fl\quad
\iota_y: \cU_y:=\fl\setminus \{x\} \cong [\bC/G_y]\hookrightarrow \fl
$$
induce surjective group homomorphisms
$$
\iota_{x*} : \pi_1(\cU_x) \cong G_x \to \pi_1(\fl),\quad
\iota_{y*} : \pi_1(\cU_y) \cong G_y \to \pi_1(\fl).
$$
\item $\iota_{x*}$ and $\iota_{y*}$ restrict to the same group homomorphism $G\to \pi_1(\fl)$, whose kernel is a cyclic group $\mu_d$ contained in the center $Z(G)$ of $G$, and whose cokernel 
is $\pi_1(\fl^{\rig})\cong \mu_a$. In other words, we have the following exact sequence of finite groups:
$$
1\to   \mu_d \to G \to \pi_1(\fl) \to \mu_a\to 1.
$$
\item We have a commutative diagram
$$
\xymatrix{
1 \ar[r] & G \ar[r]  \ar[d]^{\id_G} & G_x \ar[r]^{\iota_{x*}}  \ar[d] & \mu_{r_x} \ar[r] \ar[d]^{q_x} & 1 \\
  &  G   \ar[r]  & \pi_1(\fl) \ar[r] & \mu_a \ar[r] & 1 \\
1 \ar[r] & G\ar[r] \ar[u]_{\id_G} & G_y  \ar[r]^{\iota_{y*}} \ar[u] &  \mu_{r_y} \ar[r] \ar[u]_{q_y} &1 
} 
$$
where $\id_G:G\to G$ is the identity map, the rows are exact, and $q_x, q_y$ are surjective. The maps $\mu_{r_x} \to \Out(G)$ and $\mu_{r_y}\to \Out(G)$ factor through $\mu_a \to \Out(G)$.
\item We have $(r_x,r_y)= (ap, aq)$, where $p,q\in \bZ_{>0}$ are coprime.  The universal cover of $\fl^{\rig}$ is $\cF(p,q)=\bP(p,q)$; the universal cover of $\fl$ is the weighted projective line $\bP(d p, d q)$.  (Recall that $\bP(m,n)$ is simply connected for any positive integers $m,n$.)

\item There exist 
\begin{itemize}
\item a central extension $E$ of the finite group $\pi_1(\fl)$ by $\bC^*$, so that we have a short exact sequence of groups
$$
1\to \bC^* \stackrel{i}{\to} E \to \pi_1(\fl)\to 1
$$
where $\bC^*$ is contained in the center $Z(E)$ of $E$ and is the connected component of the identity of $E$, and
\item a representation $\rho: E\to GL(2,\bC)$,
such that $\rho\circ i(t) = (t^{d p}, t^{d q})$ and
\begin{equation}\label{eqn:C2E}
\fl\cong [ (\bC^2-\{0\})/E].
\end{equation}
\end{itemize}
The inclusion $i: \bC^* \hookrightarrow E$ induces a surjective morphism
$$
\pi:\tilde{\fl}:= \bP(dp, dq) = [(\bC^2-\{0\})/\bC^*] \longrightarrow \fl=[ (\bC^2-\{0\})/E]
$$
which is the universal covering map. Taking the rigidification yields
$$
\pi^{\rig}: \tilde{\fl}^\rig = \bP(p, q)=\cF(p,q)\longrightarrow \fl^\rig = \cF(r_x,r_y) =\cF(ap, aq)
$$
which is also the universal covering map.  
\end{enumerate}

The GKM condition implies the image of $\rho$ in (6) lies in (up to conjugation) the subgroup 
$GL(1,\bC)\times GL(1,\bC)$ of diagonal matrices, i.e. $\rho = \rho_x\oplus \rho_y$ is the direct sum of two 1-dimensional representations
of $E$. 

Under the isomorphism \eqref{eqn:C2E} we have the following identifications:
$$
x =[1,0],\quad y=[0,1], \quad, G_x =\mathrm{Ker}(\rho_x), \quad G_y = \mathrm{Ker}(\rho_y),\quad  G= \mathrm{Ker}(\rho),
$$
$$
\rho_y(G_x) =\mu_{r_x},\quad \rho_x(G_y) = \mu_{r_y}.
$$

\subsection{Normal bundle of a proper one-dimensional orbit closure} Let 
$$
\fl=[(\bC^2-\{0\})/E]
$$ 
be as in Section \ref{sec:one-dim}. The Picard group of $\fl$ is given by
\begin{equation} \label{eqn:Picl}
\mathrm{Pic}(\fl) = \Hom(E,\bC^*).
\end{equation}
The normal bundle of $\fl$ in $\cX$ is a direct sum of $(r-1)$-line bundles over $\fl$:
$$
N_{\fl/\cX} = L_1\oplus \cdots \oplus L_{r-1}.
$$
For  $i=1,\ldots, r-1$, let $\rho_i \in \Hom(E,\bC^*)$ correspond to $L_i \in \Pic(\fl)$ under the isomorphism \eqref{eqn:Picl}. Then
the total space of $L_i$ is the quotient stack
$$
L_i = [ ((\bC^2-\{0\})\times \bC )/E]
$$
where the action of $E$ is given by 
$$
g\cdot (X,Y,Z)= (\rho_x(g) X, \rho_y(g) Y, \rho_i(g) Z).
$$
If $t\in \bC^*\subset E$ then 
$$
t\cdot(X,Y,Z)= (t^{dp} X, t^{dq} Y, t^{d_i} Z)
$$
for some $d_i \in \bZ$.  Recall that for any positive integers $m,n$, 
$$
\Pic(\bP(m,n))\cong \bZ
$$
is generated by 
$$
\cO_{\bP(m,n)}(1)  = [((\bC^2-\{0\})\times \bC)/\bC^*]
$$
where $t\in \bC^*$ acts by $t\cdot (X,Y,Z) = (t^{m}X, t^{n}Y, t Z)$.  We have
$$
\langle \cO_{\bP(m,n)}(1), [\bP(m,n)^{\rig}]    \rangle = \frac{1}{l.c.m.(m,n)}
$$
where $l.c.m.(m,n)$ is the least common multiple of $m,n$.

$$
\pi^* L_i = \cO_{\bP(dp,dq)}(d_i)
$$
where $\pi:\bP(dp,dq)\to \fl$ is the universal cover.  Define 
$$
a_i:=\langle  c_1(L_i), [\fl^\rig]\rangle  = \frac{d_i}{adpq}.
$$

There is a map $\tilde{\pi}^{\rig}: \bP(dp,dq) \to \bP(p,q)$ from the universal
cover $\bP(dp, dq)$ of $\fl$ to the universal cover of $\bP(p,q)$ of $\fl^\rig$; this map can be identified with the map to rigidification, and is of degree $1/d$.
We have 
$$
(\tilde{\pi}^{\rig})^* \cO_{\bP(p,q)}(1) = \cO_{\bP(dp,dq)}(d).
$$
The map from $\fl$ to $\fl^{\rig}$ is of degree $1/|G|$. The universal covering maps $\bP(dp,dq)\to \fl$ and
$\bP(p,q)\to \fl^\rig =\cF(ap,aq)$ are of degrees $a|G|/d$ and $a$, respectively.

We have
\begin{itemize}
\item The $G_x$-actions on $T_x \fl$ and $(L_i)_x$ are given by 
$\rho_y|_{G_x}$ and $\rho_i|_{G_x}$, respectively;
\item The $G_y$-actions on $T_y\fl$ and $(L_i)_y$ are given by
$\rho_x|_{G_y}$ and $\rho_i|_{G_y}$, respectively.
\end{itemize}

For $i=1,\ldots, r-1$, let $\bw_{x,i}$ and $\bw_{y,i}$  
be the weights of the $T$-actions on $(L_i)_x$ and $(L_i)_y$, respectively;
let $\bw_{x,r}$ and $\bw_{y,r}$ be the weights of the $T$-action on $T_x\fl$ and $T_y\fl$, respectively. Then
$$
r_x \bw_{x,r} + r_y \bw_{y,r}=0.
$$
For $i=1,\dots,r$, 
$$
\bw_{y,i} = \bw_{x,i} - a_i r_x  \bw_{x,r} = \bw_{x,i} + a_i r_y \bw_{y,r}.
$$
In particular, $a_r =\frac{1}{r_x}+ \frac{1}{r_y}$. 

The total space of $N_{\fl/\cX}$ is the quotient stack 
$$
[\left((\bC^2-\{0\})\times \bC^{r-1}\right)/E]
$$
where $E$ acts on $(\bC^2 -\{0\})\times \bC^{r-1}$ linearly by 
$\rho_x\oplus \rho_y \oplus \rho_1\oplus \cdots \oplus \rho_{r-1}$. 

\begin{remark}
If $\cX$ is a smooth toric DM stack (in the sense of \cite{BCS}) then $N_{\fl/\cX}$ is a smooth toric DM substack, and the above presentation as a quotient
stack can be constructed by the stacky fan, where $E$ is abelian.
\end{remark}

We define the GKM graph of $N_{\fl/\cX}$ as follows.
\begin{itemize}
\item The underlying abstract graph is a tree with two $r$-valent vertices
$\si_x,\si_y$ connected by a compact edge $\ep$. There are $r-1$ rays $\ep_1,\ldots, \ep_{r-1}$ emanating
from the vertex $\si_x$ and $r-1$ rays $\ep_1',\ldots, \ep'_{r-1}$ emanating from the vertex $\si_y$.
\item The vertices $\si_x$ and $\si_y$ are decorated by finite groups $G_x$ and $G_y$, respectively. 
\item The edge $\ep_i$ is decorated by the kernel $G_i$ of $\phi_{x,i}:= \rho_i|_{G_x}$.
The edge $\ep'_i$ is decorated by the kernel $G_i'$ of $\phi_{y,i}:= \rho_i|_{G_y}$.
The edge $\ep$ is decorated by the group $G$.
\item The flag $(\si_x,\ep_i)$ is decorated by $\bw_{x,i}\in M_\bQ$, $\phi_{x,i}\in \Hom(G_x,\bC^*)$,
and the injection $G_i \hookrightarrow G_x$.  
The flag $(\si_y,\ep'_i)$ is decorated by $\bw_{y,i}\in M_\bQ$, $\phi_{y,i}\in \Hom(G_y,\bC^*)$,
and the injection $G'_i \hookrightarrow G_y$. 
The flag $(\si_x,\ep)$ is decorated by $\bw_x$, $\rho_y|_{G_x}$, and the injection $G\hookrightarrow G_x$.
The flat $(\si_y,\ep)$ is decorated by $\bw_y$, $\rho_x|_{G_y}$, and the injection $G\hookrightarrow G_y$.
\item The unique compact edge $\ep$ is decorated by the central extension $E$ of $\pi_1(\fl)$ by $\bC^*$ and
$\rho_x,\rho_y,\rho_1,\ldots,\rho_{r-1} \in \Hom(E,\bC^*)$ with isomorphisms $G_x\cong\Ker(\rho_x)$, $G_y\cong \Ker(\rho_y)$,
$G\cong \Ker(\rho_x)\cap \Ker(\rho_y)$.
\end{itemize}

\subsection{The stacky GKM graph of a smooth GKM stack}\label{GKM-graph}
Let $\cX$ be a smooth GKM stack of dimension $r$, 
so that $T=(\bC^*)^m$ acts algebraically on $\cX$. 
Similar to Guillemin-Zara \cite{GZ1, GZ2}, we define the stacky GKM graph of $\cX$. 
This generalizes the GKM graph of an algebraic GKM manifold in \cite[Section 2.2]{LS}
and the toric graph of a smooth toric DM stack in \cite[Section 8.6]{Liu}.

Let $V(\Up)$ (resp. $E(\Up)$)
denote the set of vertices (resp. edges) in $\Up$. 
\begin{enumerate}
\item (Vertices) We assign a vertex $\si$ to each torus fixed 
point $\fp_\si$ in $\cX$. Let $p_\si$ be the corresponding
torus fixed point in the coarse moduli space $X$.
\item (Edges) We assign an edge $\ep$ to each one-dimensional orbit $\fo_\ep$ in $X$, and choose a point $\fp_\ep$ on $\fo_\ep$. Let $\fl_\ep$ be the closure of $\fo_\ep$, and let $\ell_\ep$ be the coarse
moduli space of $\fl_\ep$. 
Let $E(\Up)_c:= \{ \ep\in E(\Up):\ell_\ep \cong \bP^1 \}$ be the set of compact edges in $\Up$. (Note that $E(\Up)_c = E(\Up)$ if $\cX$ is proper.)

\item (Flags) The set of flags in the graph $\Up$ is given by 
$$
F(\Up)= \{ (\ep,\si)\in E(\Up)\times V(\Up): \si\in \ep\}= \{(\ep,\si)\in E(\Up)\times V(\Up): p_\si\in \ell_\ep\}.
$$
\item (Inertia) For each $\sigma\in V(\Up)$, we assign a finite group $G_{\sigma}$ 
which is the inertia group of $\fp_\si$, so that $\fp_\si= \cB G_{\si}$. For each $\ep\in E(\Up)$, we assign a finite group
$G_\ep$ which is the inertial group of $\fp_\ep$ in item (2) above.

\item For every flag $(\ep,\sigma)\in F(\Up)$, we choose a path from $\fp_\ep$ to $\fp_\si$, which determines an injective group homomorphism
$j_{(\ep,\sigma)}: G_\ep \to G_\si$. Let $\phi_{(\ep,\sigma)}: G_\si \to \bC^*$ be the group homomorphism which corresponds 
to the 1-dimensional $G_\si$ representation $T_{\fp_\si}\fl_\ep$. 
The image of $\phi_{(\ep,\sigma)}$ is a finite cyclic group; let
$r_{(\ep,\si)}$ be the cardinality of this finite cyclic group. We have the following
short exact sequence of finite groups:
$$
1\to G_{\ep} \stackrel{j_{(\ep,\si)}}{\longrightarrow} G_{\sigma} 
\stackrel{\phi_{(\ep,\si)}}{\longrightarrow} \mu_{r_{(\ep,\si)}} \to 1.
$$ 
So 
$$
r_{(\ep,\si)}=\frac{|G_{\sigma}|}{|G_{\ep}|}.
$$
\item (fundamental groups) For each compact edge $\ep\in E(\Up)_c$, there is a group homomorphism
$G_\ep \to \pi_1(\fl_\ep)$
whose kernel is a cyclic subgroup of  the center $Z(G_\ep)$ of $G_\ep$. Let $d_\ep$ be the cardinality of this cyclic subgroup. Then we have
an exact sequence  of finite groups:
$$
1\to \mu_{d_\ep}  \to G_\ep \to \pi_1(\fl_\ep) \to \pi_1(\fl_\ep^\rig) \to 1.
$$
\item (central extension of fundamental groups) For each compact edge $\ep\in E(\Up)_c$, let $\si_x,\si_y\in V(\Up)$ be the two ends of $\ep$, and let
$$
x = p_{\si_x},\quad y= p_{\si_y}, \quad r_x = r_{(\ep,\si_x)},\quad r_y = r_{(\ep,\si_y)},\quad a_\ep = g.c.d. (r_x,r_y).
$$
Then $\fl_\ep^\rig$ is the football $\cF(r_x,r_y)$, so that $\pi_1(\fl^\rig_\ep)=\mu_{a_\ep}$. 
There is a triple $(i_\ep, E_\ep,\rho_\ep)$, where $i_\ep:\bC^*\hookrightarrow E_\ep$ is a central injection and 
$E_\ep/\bC^* \cong \pi_1(\fl_\ep)$, $\rho_\ep = (\rho_x, \rho_y): E_\ep\to \bC^* \times \bC^*$ is a group homomorphism
and 
$$
\rho_\ep\circ i_\ep = (t^{d_\ep r_x/a_\ep}, t^{d_\ep r_y/a_\ep}).
$$  
We have an isomorphism  $[(\bC^2-\{(0,0)\}) /E_\ep]\cong \fl_\ep$, and under this isomorphism
$$
x=[1,0],\quad y=[0,1],\quad G_{\si_x}= \Ker(\rho_x),\quad G_{\si_y} =\Ker(\rho_y),\quad  G_\ep = \Ker(\rho).
$$

\item (connection) Let $\ep\in E(\Up)_c$, and let $\si_x$ and $\si_y$ be as above. Let
$E_x$ and $E_y$ be the set of edges emanating from $\si_x$ and $\si_y$, respectively.
The normal bundle $N_{\fl_\ep/\cX}$ of $\fl_\ep$ in 
$\cX$ is a direct sum of line bundles
$$
N_{\fl_\ep/\cX}= L_1\oplus \cdots \oplus L_{r-1}.
$$
For $i=1,\ldots, r-1$ there exist $\ep_i \in E_x$ and $\ep_i'\in E_y$ such that 
$(L_i)_x = T_x \fl_{\ep_i}$ and $(L_i)_y = T_y \fl_{\ep_i'}$. Then 
$$
E_x=\{\ep_1,\ldots, \ep_{r-1},\ep \},\quad E_y=\{ \ep'_1,\ldots, \ep'_{r-1},\ep\}.
$$
Define a bijection $\theta_{(\ep, \si_x)}: E_x \to E_y$ by sending $\ep_i$ to $\ep_i'$ and sending $\ep$ to $\ep$;
let $\theta_{(\ep,\si_y)}:E_y\to E_x$ be the inverse map. We say $\{ \ep_i, \ep'_i\}$ is a  pair of edges
related by the parallel transport along the compact edge $\ep$. There exists $\rho_i \in \Hom(E_\ep,\bC^*)$ such that 
$L_i = [((\bC^2-\{(0,0)\})\times \bC)/E]$ where $E$ acts by $\rho_x\oplus \rho_y \oplus \rho_i$.  Then 
$\rho_i\circ i_\ep:\bC^*\to \bC^*$ is given by $t\mapsto t^{d_i}$ for some $d_i\in \bZ$ and
$$
a_i:=\langle c_1(L_i),[\fl_\ep^\rig]\rangle = \frac{d_i}{l.c.m.(r_x,r_y)d_\ep}=\frac{d_i a_\ep}{r_x r_y d_\ep} \in \bQ.
$$
(Note that, if $\fl_\ep$ is the projective line $\bP^1$ then $r_x=r_y=a_\ep=d_\ep=1$, so $a_i=d_i\in \bZ$.) 
Let $\Delta_\ep$ be the set of pairs of edges related by the parallel transport along the compact edge $\ep$.
For each pair $\delta \in \Delta_\ep$,  we get $\rho_\delta \in \Hom(E_\ep,\bC^*)$ 
associated to a line bundle over $\fl_\ep$ which is a summand of $N_{\fl/\cX}$.

\item (compatibility along compact edges) 
$$
\rho_i|_{G_{\si_x}}=\phi_{(\ep_i,\si_x)},\quad 
\rho_i|_{G_{\si_y}}=\phi_{(\ep'_i,\si_x)},\quad
\rho_y|_{G_{\si_x}}=\phi_{(\ep, \si_x)},\quad
\rho_x|_{G_{\si_y}}=\phi_{(\ep, \si_y)}.
$$
\end{enumerate}

Assumption \ref{assume} can be rephrased in terms of the graph $\Up$ as follows.
\begin{assumption}
\begin{enumerate}
\item $V(\Up)$ is non-empty.
\item Each edge in $E(\Up)$ contains at least one vertex.
\end{enumerate}
\end{assumption}

Given a vertex $\si\in V(\Up)$, we denote by $E_\si$ the set of edges containing
$\si$, i.e. $E_\si:=\{e\in E:(\ep,\si)\in F(\Up)\}$. Then $|E_\si|=r$ for
all $\si\in V(\Up)$, so $\Up$ is an $r$-valent graph.

Given a flag $(\ep,\si)\in F(\Up)$, let 
$\bw_{(\ep,\si)} \in M_{\mathbb{Q}}$
be the weight of $T$-action on $T_{\fp_\si} \fl_\ep$, the tangent
line to $\fl_\ep$ at the fixed point $\fp_\si=\cB G_\si$, namely,
$$
\bw_{(\ep,\si)} := c_1^T(T_{\fp_\si}\fl_\ep) \in H^2_T(\fp_\si;\bQ)\cong M_{\bQ}. 
$$

This gives rise to a map 
$$
\bw:F(\Up)\longrightarrow M_{\bQ},\quad (\ep,\si)\mapsto \bw_{(\ep,\si)} 
$$
satisfying the following
properties.
\begin{enumerate}
\item (GKM hypothesis) Given any $\si\in V(\Up)$, and any 
two distinct edges $\ep, \ep'\in E_\si$, 
$\bw_{(\ep,\si)}$ and $\bw_{(\ep',\si)}$
are linearly independent in $M_\bQ$.

\item (integrality) For any flag $(\ep,\si)\in F(\Up)$, 
$\overline{\bw}_{(\ep,\si)}:= r_{(\ep,\si)} \bw_{(\ep,\si)}\in M$.

\item Suppose that $\ep\in E(\Up)_c$ is a compact edge and $\si_x, \si_y\in V(\Up)$ are its two ends.  
\begin{enumerate}
\item $r_{(\ep,\si_x)}\bw_{(\ep,\si_x)} + r_{(\ep,\si_y)}\bw_{(\ep,\si_y)}=0$, i.e.
$ \overline{\bw}_{(\ep,\si_x)} + \overline{\bw}_{(\ep,\si_y)}=0$.
\item Let $E_{\si_x}= \{ \ep_1,\ldots, \ep_r\}$, where
$\ep_r=\ep$, and let $\ep_i':=\theta_{(\si_x,\ep)}(\ep_i) \in E_{\si_y}$. 
Then
$$
\bw_{(\ep_i',\si_y)} =\bw_{(\ep_i,\si_x)}- a_i r_{(\ep,\si_x)} \bw_{(\ep,\si_x)}=\bw_{(\ep_i, \si_x)}+a_i r_{(\ep, \si_y)}\bw_{(\ep, \si_y)}.
$$
\end{enumerate}
\end{enumerate}
The normal
bundle of the 1-dimensional smooth DM stack $\fl_\ep$ in $\cX$ is given by
$$
N_{\fl_\ep/\cX}\cong L_1\oplus\cdots \oplus  L_{r-1}
$$
where $L_i$ is a degree $a_i$ $T$-equivariant line bundle
over $\fl_\ep$ such that the weights of the
$T$-action on the fibers $(L_i)_z$ and $(L_i)_y$
are $\bw_{(\ep_i,\si_x)}$ and $\bw_{(\ep_i',\si_y)}$, respectively.  
The map $\bw:F(\Up)\to M_{\bQ}$ is called the {\em axial function}.

We call $\vUp$, which is the abstract graph $\Up$ together with the above decorations and constraints, 
the {\em stacky GKM graph} of the smooth GKM stack $\cX$ with the $T$-action. 

If $\rho: T'\to T$ is a homomorphism between complex algebraic tori, then
$T'$ acts on $X$ by $t'\cdot x = \rho(t')\cdot x$, where $t'\in T'$, 
$\rho(t')\in T$, $x\in X$. If the zero-dimensional and one-dimensional orbits of this
$T'$-action coincide with those of the $T$-action, then the GKM graph with this $T'$-action is obtained by replacing
$\bw_{(\ep,\si)}\in M_\bQ$ by $\rho^*\bw_{(\ep,\si)}\in M'_\bQ$, where 
$$
\rho^*:  M_\bQ = H^2(BT;\bQ) \to M'_\bQ := H^2(BT';\bQ).
$$

\subsection{Equivariant Chen-Ruan orbifold cohomology group}
Let $\cX$ be a smooth GKM stack. The $T$-action on $\cX$ induces a $T$-action on 
its inertia stack $\cIX=\bigsqcup_{i\in I}\cX_i$ and on each $\cX_i$.

Let
$$
R_T:= H_T^*(\pt;\bQ)=H^*(BT;\bQ) =\bQ[u_1,\ldots,u_m]
$$
where $\deg(u_i)=2$; let $Q_T=\bQ(u_1,\ldots,u_m)$ be its fractional field.

As a graded $\bQ$-vector space,  $T$-equivariant Chen-Ruan orbifold cohomology group of an smooth GKM stack is defined to be
$$
H^*_{\CR, T} (\cX;\bQ): = \bigoplus_{a\in \bQ_{\geq 0}}H^a_{\CR,T}(\cX;\bQ)
$$
where
$$
H^a_{\CR,T}(\cX;\bQ)=\bigoplus_{i\in I} H^{a-2\age(\cX_i)}_T(\cX_i;\bQ).
$$

Suppose that $\cX$ is proper. For each $i\in I$, we have the following proper pushforward to a point:
$$
\int_\cX: H^*_T(\cX_i;\bQ) \longrightarrow H^*_T(\pt;\bQ) = R_T
$$
which is $R_T$-linear.
The $T$-equivariant orbifold Poincar\'{e} pairing is defined by 
\begin{equation}\label{eqn:T-pairing}
(\alpha,\beta)_T := 
\begin{cases} 
\int_{\cX_i} \alpha \cup \iota_i^* \beta, & j=\iota(i),\\
0, & j\neq \iota(i),
\end{cases}
\end{equation}
where $\alpha\in H^*_T(\cX_i;\bQ)$, $\beta\in H^*_T(\cX_j;\bQ)$. 
 
When $\cX$ is not proper, we define a $T$-equivariant Poincar\'{e} pairing on 
$$
H_{\CR,T}^*(\cX;Q_T)= H_{\CR,T}^* (\cX;\bQ)\otimes_{R_T} Q_T
$$
as follows:
\begin{equation}
(\alpha,\beta)_T := 
\begin{cases} 
\displaystyle{
\int_{\cX_i^T}\frac{(\alpha \cup \iota_i^* \beta)|_{\cX_i}}{e_T(N_{\cX_i^T/\cX_i})}, 
} & j=\iota(i),\\
0, & j\neq \iota(i),
\end{cases}
\end{equation}
where $\cX_i^T\subset \cX_i$ is the $T$ fixed substack, and $e_T(N_{\cX_i^T/\cX_i})$ is the $T$-equivariant Euler class of
the normal bundle $N_{\cX_i^T/\cX_i}$ of  $\cX_i^T$ in $\cX_i$. Each $\cX_i^T$ is a disjoint of union 
of finitely many (stacky) points.

\begin{example}[affine smooth GKM stack] \label{ex:affine}
Let $\cX=[\bC^r/G]$ be an affine smooth GKM stack. Let $\phi_i:G\to \bC^*$, $\bw_i\in M_\bQ$, and $r_i$  be defined as in 
Section \ref{sec:zero-dim}. Given $h\in G$, let $c_i(h)$ be the unique element in 
$$
\left\{ 0, \frac{1}{r_i},\ldots, \frac{r_i-1}{r_i} \right\}
$$ 
such that 
$$
e^{2\pi\sqrt{-1}c_i(h)} = \phi_i(h). 
$$
Then 
$$
\cI\cX = \bigsqcup_{c \in \Conj(G)} \cX_c,
$$
where
$$
\cX_c \cong [(\bC^r)^h/C_G(h)]
$$
for any $h\in c$.  We have
$$
\age(\cX_c) =\sum_{i=1}^r c_i(h)
$$
where $h$ is any element in the conjugacy class $c$.

Let $\one_c$ denote the identity element of $H_T^*(\cX_c;\bQ)$. Then
$$
H^*(\cX_c;\bQ) =\bQ \one_c,\quad
H_T^*(\cX_c;\bC) = R_T \one_c.
$$
So
$$
H_{\CR}(\cX;\bQ) = \bigoplus_{c\in \Conj(G)}  \bQ \one_c         
$$
as a $\bQ$ vector space, and
$$
H_{\CR,T}(\cX;\bQ) = \bigoplus_{c\in \Conj(G)} R_T \one_c
$$
as an $R_T$-module.

Given $c\in \Conj(G)$, define
$$
\be_c:= e_T(T_{[0/G]}(\bC^r)^h)=\prod_{i=1}^r \bw_i^{\delta_{c_i(h),0}}
$$ 
where $h$ is any element in $c$. Note that the right hand side of the above equation does not depend on the choice of $h\in c$.

Given $h\in G$, let $[h]=\{aha^{-1}:a\in G\}$ be the conjugacy class of $h$. 

The $T$-equivariant Poincar\'{e} pairing on 
$$
H_{\CR,T}(\cX;Q_T) =\bigoplus_{c \in \Conj(G)} Q_T \one_c
$$
is given by 
$$
( \one_{[h]}, \one_{[h']})_T = \frac{1}{|C_G(h)|}\frac{\delta_{[h^{-1}], [h']}}{\be_{[h]}}.
$$
\end{example}

\begin{definition}[equivariant formality]
Let $\cX$ be a smooth GKM stack, so that $T=(\bC^*)^m$ acts algebraically on $\cX$. We say
$\cX$ is equivariantly formal if
$$
H^*_{\CR,T}(\cX;\bQ) \to H^*_{\CR}(\cX;\bQ)
$$
is surjective. 
\end{definition}

Smooth toric DM stacks and affine smooth GKM stacks are equivariantly formal smooth GKM stacks.

\subsection{Cup product} \label{sec:product}
In this section, we describe the cup product on 
$$
H^*_{\CR,T}(\cX;Q_T),
$$
first for an affine smooth GKM stack, and then for any equivariantly formal smooth GKM stacks.

Given $c, c'\in \Conj(G)$, define
$$
c_i(c,c'): = c_i(h) + c_i(h')-c_i(hh') \in \{0,1\}.
$$
where $h\in c$ and $h'\in c'$; note that the right hand side of the above equation does not depend on choice of $h\in c$ and $h'\in c'$.
\begin{itemize}
\item Let $\cX=[\bC^r/G]$ be an affine smooth GKM stack as in Example \ref{ex:affine}. The cup product on $H_{\CR,T}^*(\cX;\bQ)$ is given by 
$$
\one_c \star \one_{c'} = \prod_{i=1}^r \bw_i^{c_i(c,c')} \sum_{h\in c, h'\in c'}
\frac{|C_G(hh')| }{|G|} \one_{[hh']}.
$$

\item Let $\cX$ be an equivariantly formal smooth GKM stack, 
and let $\vUp$ be the stacky GKM graph of $\cX$. Then we have an isomorphism of 
$Q_T$-algebras
\begin{equation}\label{eqn:CR-QT}
H_{\CR,T}^*(\cX;Q_T) \cong \bigoplus_{\si\in V(\Up)} H_{\CR,T}^*(T_{\fp_\si} \cX ;Q_T)
\end{equation}
which preserves the $T$-equivariant Poincar\'{e} pairing; the isomorphism \eqref{eqn:CR-QT} is an isomorphism of Frobenius algebras over the field $Q_T$. 
\end{itemize}

\section{Abstract stacky GKM graphs and formal smooth GKM stacks}\label{sec:abstract-formal}

Let $\cX$ be a smooth GKM stack equipped with a $T$-action. The formal completion $\hat{\cX}$ of $\cX$ along its
1-skeleton $\cX^1$, together with the $T$-action inherited from $\cX$, can be reconstructed from 
the stacky GKM graph of $\cX$. In this section,  we will define abstract stacky GKM graphs which are generalization of 
stacky GKM graphs of smooth GKM stacks.  Given an abstract stacky GKM graph, we will construct a formal smooth GKM stack, 
which is a formal smooth DM stack together with an action by an algebraic torus $T=(\bC^*)^m$. The construction of a formal
smooth GKM stack from an abstract stacky GKM graph can be viewed as generalization of the reconstruction of $\hat{\cX}$ from
the stacky GKM graph of a smooth GKM stack $\cX$, and is inspired by the construction of a formal toric Calabi-Yau (FTCY) 
threefold from 
an FTCY graph in \cite[Section 3]{LLLZ}.

\subsection{Abstract stacky GKM graphs}\label{sec:abstract}
We fix $T=(\bC^*)^m$ and a positive integer $r$. An {\em abstract stacky GKM graph} is a decorated graph consisting of the following data. 
\begin{enumerate}
\item (graph)
The underlying graph $\Gamma$ is a connected $r$-valent graph $\Gamma$ with finitely many vertices and edges. 
Let $V(\Up)$ (resp. $E(\Up)$) denote the set of vertices (resp. edges) in $\Gamma$.
Each edge in $E(\Up)$ is either a compact edge connecting two vertices or a ray emanating from one vertex.
Let $E(\Up)_c\subset E(\Up)$ be the set of compact edges. Let 
$$
F(\Up)=\{ (\ep,\si)\in E(\Up)\times V(\Up): \si\in \ep\}
$$
be the set of flags in $\Gamma$. Given a vertex $\si\in V(\Up)$, let 
$$
E_\si:= \{ \ep \in E(\Up): (\ep,\si)\in F(\Up)\}
$$ 
be the set of edges emanating from the vertex $v$. By the $r$-valent condition, 
$|E_\si|=r$ for all $r\in V(\Up)$.

\item (inertia and tangent representations) 
Each vertex $\si \in V(\Up)$ (resp. edge $\ep \in E(\Up)$) is decorated by a finite group
$G_\si$ (resp. $G_\ep$). Each flag $(\ep,\si)\in F(\Up)$ is decorated by 
\begin{itemize}
\item an injective group homomorphism $j_{(\ep,\si)}:G_\ep \hookrightarrow G_\si$, and
\item a one-dimensional representation  $\phi_{(\ep,\si)}: G_\si\to GL(1,\bC)= \bC^*$, 
\end{itemize}
such that $\mathrm{Im}(j_{(\ep,\si)}) = \Ker (\phi_{(\ep,\si)})$. \\
Note that the image of 
$\phi_{(\ep,\si)}$ is a finite cyclic group; let $r_{(\ep,\si)}$ be the cardinality of
this finite cyclic group. Then we have a short exact sequence of finite groups:
$$
1\to G_\ep\stackrel{j_{(\ep,\si)}}{\longrightarrow} G_\si \stackrel{\phi_{(\ep,\si)} }{\longrightarrow}  
\mu_{r_{(\ep,\si)}} \to 1.
$$

\item (fundamental groups and central extensions) Let $\ep\in E(\Up)_c$ be a compact edge, and let $\si_x, \si_y
\in V(\Up)$ be the two ends of $\ep$. Let $a_\ep = g.c.d. (r_{(\ep,\si_x)}, r_{(\ep,\si_y)})$. In addition 
to $G_\ep$, $\ep$ is decorated by:
\begin{itemize} 
\item Another finite group $\pi_\ep$ together with a group homomorphism 
$G_\ep \to \pi_\ep$ such that we have the following exact sequence of finite groups:
$$
1\to \mu_{d_\ep} \to G_\ep \to \pi_\ep \to \mu_{a_\ep}\to 1
$$
where $\mu_{d_\ep}$ is contained in the center of $G_\ep$. 
\item A central extension $1\to \bC^* \stackrel{i_\ep}{\to} E_\ep \to \pi_\ep \to 1$ of $\pi_\ep$ by $\bC^*$,
and a group homomorphism $\rho_\ep = (\rho_x,\rho_y): E_\ep\to \bC^*\times \bC^*$. 
\end{itemize}

\item (connection) Let $\ep \in E_c(\Up)$ be a compact edge, and let $\si_x, \si_y \in V(\Up)$ be the two ends
of $\ep$. There are bijections $\theta_{(\ep,\si_x)}: E_{\si_x}\to E_{\si_y}$ and
$\theta_{(\ep,\si_y)}: E_{\si_y}\to E_{\si_x}$ which are inverses of each other and send $\ep$ to $\ep$. 

\item (normal representations) Suppose that $\ep\in E_c(\Up)$ is a compact edge, $\si_x,\si_y\in V(\Up)$ are two ends of 
$\ep$, and $\delta=\{ \ep_x, \ep_y \}$ is a pair of edges such that $\ep_x \in E_{\si_x}-\{\ep\}$ and $\ep_y =\theta_{(\ep,\si_x)}(\ep_x) \in E_{\si_y}-\{\ep\}$.
Such a pair is decorated by a one-dimensional representation $\rho_\delta: E_\ep\to GL(1,\bC)=\bC^*$.

\item (compatibility along compact edges) In the notation of (3), (4), (5) above, $\Ker(\rho_x)=G_{\si_x}$, $\Ker(\rho_y)=G_{\si_y}$, $\Ker(\rho_\ep)=G_\ep$,
$$
\rho_y|_{G_{\si_x}} = \phi_{(\ep,\si_x)},\quad
\rho_x|_{G_{\si_y}} = \phi_{(\ep,\si_y)},\quad
\rho_\delta|_{G_{\si_x}} = \phi_{(\ep_x,\si_x)},\quad
\rho_\delta|_{G_{\si_y}} =\phi_{(\ep_y,\si_y)}.
$$

\item (axial function) There is a map 
$$
\bw: F(\Up)\to M_\bQ,\quad (\ep,\si)\mapsto \bw_{(\ep,\si)}  
$$
satisfying the following properties.
\begin{enumerate} 
\item (GKM hypothesis) Given any $\si\in V(\Up)$ and any two distinct edges $\ep,\ep'\in E_\si$,
$\bw_{(\ep,\si)}$ and $\bw_{(\ep',\si)}$ are linearly independent vectors in $M_\bQ$.

\item (integrality)  For any $(\ep,\si)\in F(\Up)$, $\overline{\bw}_{(\ep,\si)}:= r_{(\ep,\si)}\bw_{(\ep,\si)}\in M$.

\item For any compact edge $\ep\in E(\Up)_c$, let $\si_x, \si_y \in V(\Up)$ be its two ends. Then the following properties hold. 
\begin{enumerate}
\item $r_{(\ep,\si_x)} \bw_{(\ep,\si_x)} + r_{(\ep,\si_y)}\bw_{(\ep,\si_y)}=0$, i.e.,
$\overline{\bw}_{(\ep,\si_x)} + \overline{\bw}_{(\ep,\si_y)} =0$.
\item Suppose that $E_\si = \{ \ep_1,\ldots, \ep_r\}$, where $\ep_r=\ep$. Let
$\ep_i':= \theta_{(\ep,\si)}(\ep_i)\in E_{\si'}$, so that $E_{\si'}=\{\ep_1',\ldots, \ep_r'\}$. Let
$$
a_i = \frac{d_i a_\ep}{r_{(\ep,\si_x)}r_{(\ep,\si_y)}d_\ep}
$$
where $d_i\in \bZ$ is determined by by $\rho_{\{ \ep_i,\ep'_i\}}\circ i_\ep(t)=t^{d_i}$ for $t\in \bC^*$.  Then 
$$
\bw_{(\ep_i',\si_y)} = \bw_{(\ep_i,\si_x)}-a_i r_{(\ep,\si_x)} \bw_{(\ep,\si_x)}
=\bw_{(\ep_i,\si_i)}+a_i r_{(\ep,\si_y)} \bw_{(\ep,\si_y)},
$$ 
or equivalently, 
$$
\bw_{(\ep_i',\si_y)}= \bw_{(\ep_i,\si_x)}-a_i \overline{\bw}_{(\ep,\si_x)}
=\bw_{(\ep_i,\si_i)}+a_i \overline{\bw}_{(\ep,\si_y)},
$$
In particular, $\ep_r'=\ep_r=\ep$ and $a_r = \frac{1}{r_{(\ep,\si_x)}} +\frac{1}{r_{(\ep,\si_y)}}$.  
\end{enumerate}
\end{enumerate}

Let $\vUp$ denote the underlying abstract graph $\Up$ together with all the above decorations.
\end{enumerate}

\begin{remark} We may also define abstract GKM graphs by the following specialization. 
\begin{itemize}
\item All the finite groups $G_\si$, $G_\ep$, $\pi_\ep$ are trivial,  and we always have $E_\ep =\bC^*$ and
$\rho_x, \rho_y:\bC^*\to \bC^*$ are identity maps.  So we do not need (2), (3), (6) above.
\item In (7), the axial function $\bw$ takes value in $M$ instead of $M_\bQ$, and
$$
r_{(\ep,\si)}=r_{(\ep,\si')}=1, \quad a_\ep = d_\ep =1, \quad a_i=d_i \in \bZ. 
$$
\item The normal characters in (5) are determined by the axial function.
\end{itemize}
Abstract GKM graphs are generalization of GKM graphs of algebraic GKM manifolds \cite[Section 2.2]{LS}.
\end{remark}

\subsection{Formal smooth GKM stacks}

Given an abstract stacky GKM graph $\vUp$ defined as in the previous subsection, we will construct a formal smooth DM stack $\hat{\cX}_{\vUp}$ of dimension $r$ equipped with an action of $T=(\bC^*)^m$.

For any flag $(\ep,\si)\in F(\Up)$, define a ``stacky'' affine line 
$$
D_{(\ep,\si)}:= [ \Spec\bC[z_{(\ep,\si)}]/G_\si] \cong [\bA^1/G_\si]
$$
where $G_\si$ acts on $\bA^1$ via the group homomorphism $\phi_{(\ep,\si)}:G_\si\to \bC^*$. 
The coarse moduli space of $D_{(\ep,\si)}$ is 
$$
\Spec(\bC[z_{(\ep,\si)}]^{G_\si}) =\Spec(\bC[z_{(\ep,\si)}^{r_{(\ep,\si)}}]) =
\Spec(\bC[x_{(\ep,\si)}]) \cong \bA^1
$$
where $x_{(\ep,\si)}= (z_{(\ep,\si)})^{r_{(\ep,\si)}}$.

For any vertex $\si\in V(\Up)$, define an $r$-dimensional affine smooth GKM stack
$$
\cX_\si=  [\Spec \bC[z_{(\ep,\si)}: \ep\in E_\si] /G_\si] = [\bA^{E_v}/G_\si].
$$
The $T$-action on $z_{(\ep,\si)}$ is determined by $\bw_{(\ep,\si)}\in M_\bQ$. 
Let $\hat{\cX}_\si$ be the formal completion of $\cX_\si$ along its 1-skeleton.

For any compact edge $\ep\in E(\Up)_c$, define
$$
\fl_\ep:=[(\bC^2-\{0\})/E_\ep]
$$
where the action of $E_\ep$ is given by the group homomorphism $\rho_\ep: E_\ep\to \bC^*\times \bC^*$.
Let $\si_x,\si_y\in V(\Up)$ be its two ends. 
Suppose that $E_{\si_x}=\{\ep_1,\ldots, \ep_{r-1},\ep\}$, and let
$\ep'_i=\theta_{(\ep,\si_x)}(\ep_i) \in E_{\si_y}-\{ \ep\}$. Let
$\rho_i = \rho_{\{\ep_i,\ep'_i\}} \in \Hom(E_\ep,\bC^*)$. Let $L_i$ be the line bundle over the
smooth DM curve $\fl_\ep$ defined by 
$$
L_i = [\left( (\bC^2-\{0\})\times \bC\right)/E_\ep]
$$
where the action on the last factor $\bC$ is given by the group homomorphism $\rho_i:E_\ep\to \bC^*$.
Let $\cX_\ep$ be the total space of $L_1\oplus \cdots \oplus L_{r-1}$, which is a smooth GKM stack.
Let $\hat{\cX}_\ep$ be the formal completion of $\cX_\ep$ along its 1-skeleton $\cX^1_\ep$.
By compatibility along compact edges, there are  $T$-equivariant open embeddings of formal smooth DM stacks: 
$$
i_{(\ep,\si_x)}: \hat{\cX}_{\si_x}\hookrightarrow \hat{\cX}_\ep,\quad 
i_{(\ep,\si_y)}: \hat{\cX}_{\si_y}\hookrightarrow \hat{\cX}_\ep.
$$

For any $\si\in V(\vUp)$, define 
$$
\hat{i}_\si:=\prod_{\ep\in E_\si\cap E(\Up)_c} i_{(\ep,\si)}: \hat{\cX}_\si
\longrightarrow \bigcup_{\ep\in E(\Up)_c} \hat{\cX}_\ep. 
$$
The formal DM stack $\hXY$ is defined to be the fiber product of the maps
$$
\Big\{ \hat{i}_\si: \hat{\cX}_\si \longrightarrow\bigcup_{\ep\in E(\Up)_c}\hat{\cX}_\ep \mid \si\in V(\Up)\Big\}. 
$$

If $\vUp$ is the stacky GKM graph of an smooth GKM stack $\cX$ then $\hXY$ is the formal completion of $\cX$ along
its 1-skeleton $\cX^1$.

\subsection{Equivariant Chen-Ruan orbifold cohomology of an abstract stacky GKM graph}
Given a stacky GKM graph $\vUp$,we define
$$
\cH_{\vUp} := \bigoplus_{\si \in V(\Up)}   H^*_{\CR,T}(\cX_\si; Q_T)
$$
as a Frobenius algebra over the field $Q_T$. By Section \ref{sec:product}, if $\vUp$ is the stacky GKM graph 
of an equivariantly formal smooth GKM stack $\cX$ then 
$$
\cH_{\vUp} = H^*_{\CR,T}(\cX;Q_T).
$$

\section{Orbifold Gromov-Witten theory}\label{sec:orbGW}
In \cite{CR2}, Chen-Ruan developed Gromov-Witten theory for symplectic orbifolds. 
The algebraic counterpart, Gromov-Witten theory for smooth DM stacks, was developed by Abramovich-Graber-Vistoli \cite{AGV1, AGV2}. 
In this section, we give a brief review of algebraic orbifold Gromov-Witten theory, following \cite{AGV2}.

\subsection{Twisted curves and their moduli} \label{sec:tw-curves}

An $n$-pointed, genus $g$  twisted curve is a connected proper one-dimensional
DM stack $\cC$ together with $n$ disjoint closed substacks $\fx_1,\ldots,\fx_n$ of $\cC$, such that
\begin{enumerate}
\item  $\cC$ is \'{e}tale locally a nodal curve; 
\item formally locally near a node, $\cC$ is isomorphic
to the quotient stack
$$
[\Spec(\bC[x,y]/(xy))/\mu_r],
$$ 
where the action of $\zeta\in \mu_r$ is given by $\zeta\cdot (x,y)= (\zeta x, \zeta^{-1}y)$;
\item each $\fx_i\subset \cC$ is contained in the smooth locus of $\cC$;
\item each stack $\fx_i$ is an  \'{e}tale gerbe over $\Spec\bC$ {\em with a section} (hence
trivialization); 
\item $\cC$ is a scheme outside the twisted points $\fx_1,\ldots, \fx_n$ and the singular locus;
\item the coarse moduli space $C$ is a nodal curve of arithmetic genus $g$.
\end{enumerate}
Let $\pi:\cC\to C$ be the projection to the coarse moduli space, and let
$x_i=\pi(\fx_i)$. Then $x_1,\ldots, x_n$ are distinct smooth points of
$C$, and $(C,x_1,\ldots, x_n)$ is an $n$-pointed, genus $g$ prestable curve.

Let $\cM^\tw_{g,n}$ be the moduli of $n$-pointed, genus $g$
twisted curves. Then $\cM^\tw_{g,n}$ is a smooth
algebraic stack, locally of finite type \cite{Ol}.

\subsection{Riemann-Roch theorem for twisted curves} \label{sec:orbRR}
Let $(\cC,\fx_1,\ldots, \fx_n)$ be an $n$-pointed, genus $g$ twisted curve, and
let $(C,x_1,\ldots, x_n)$ be the coarse curve, which is an $n$-pointed, genus
$g$ prestable curve.  Let $\cE \to \cC$ be a vector bundle over $\cC$. 
Then $\fx_i\cong \cB\mu_{r_i}$. There is a unique generator $\zeta$ of the cyclic group $\mu_{r_i}$ such that $\zeta$
acts on the tangent line $T_{\fx_i}\cC$ with eigenvalue $e^\frac{2\pi\sqrt{-1}}{r_i}$. Then
$\zeta$ acts on $\cE|_{\fx_i}$ with eigenvalues $e^{ \frac{2\pi\sqrt{-1}}{r_i} l_1 },\ldots, e^{ \frac{2\pi\sqrt{-1}}{r_i} l_N }$,
where $N=\rank \cE$ and  $l_i\in \{0,1,\ldots, r_i-1\}$. In other words,
$$
\cE|_{\fx_i} = \bigoplus_{i=1}^N (T_{\fx_i}\cC)^{\otimes l_i}
$$
as vector bundles over $\fx_i=\cB \mu_{r_i}$. Note that $l_1,\ldots, l_N$ are unique
up to permutation, so 
$$
\age_{x_i}(\cE) := \frac{l_1+\cdots + l_N}{r_i} \in \bQ
$$
is well-defined.  The Riemann-Roch theorem for twisted curves says
\begin{equation}\label{eqn:twistedRR}
\chi(\cE)=\int_{\cC}c_1(\cE) + \rank(\cE)(1-g) -\sum_{i=1}^n \age_{x_i}(\cE).
\end{equation}


\subsection{Moduli of twisted stable maps} \label{sec:twisted}
Let $\cX$ be a smooth DM stack with a quasi-projective
coarse moduli space $X$, and let $\beta\in H_2(X;\bZ)$ be an effective curve class in 
$X$.  An $n$-pointed, genus $g$, degree $\beta$ twisted
stable map to $\cX$ is a representable morphism $f:\cC\to \cX$, where
the domain $\cC$ is an $n$-pointed, genus $g$ twisted curve,
and the induced morphism $C\to X$ between the coarse moduli spaces
is an $n$-pointed, genus $g$, degree $\beta$ stable map to $X$.

Let $\MgcX$ be the moduli stack of $n$-pointed, genus $g$, degree $\beta$ twisted
stable maps to $\cX$. Then $\MgcX$ is a DM stack;  it is proper if $X$ is projective.

For $j=1,\ldots, n$, there are evaluation maps
$$
\ev_j:\MgcX\to \cIX = \bigsqcup_{i\in I}\cX_i
$$
where $\{\cX_i: i\in I\}$ are connected components of $\cI\cX$.
Given $\vi=(i_1,\ldots,i_n)$, where $i_j\in I$, define
$$
\MgXi :=\bigcap_{j=1}^n \ev_j^{-1}(\cX_{i_j}). 
$$
Then $\MgXi$ is a union of connected components
of $\MgcX$, and  $$
\MgcX=\bigsqcup_{\vi\in I^n}\MgXi.
$$
\begin{remark}
In the definition of twisted curves in Section \ref{sec:tw-curves}, if we replace (4) by
\begin{enumerate}
\item[(4)'] each stack $\fx_i$ is an  \'{e}tale gerbe over $\Spec\bC$;
\end{enumerate}
i.e. without a section, then the resulting moduli space is $\cK_{g,n}(\cX,\beta)$ in \cite{AGV2}, 
and the evaluation maps take values in the rigidified inertial stack $\overline{\cI}{\cX}$ instead of the inertia
stack $\cI\cX$.
\end{remark}

\subsection{Obstruction theory and virtual fundamental classes}

The tangent space $T^1_\xi$ and the obstruction space $T^2_\xi$ at
a moduli point $\xi=[f:(\cC, \fx_1,\ldots, \fx_n)\to \cX] \in \MgcX$ fit
in the {\em tangent-obstruction exact sequence}:
\begin{equation}\label{eqn:tangent-obstruction}
\begin{aligned}
0 \to& \Ext^0_{\cO_\cC}(\Omega_\cC(\fx_1+\cdots + \fx_n) , \cO_\cC)\to H^0(\cC,f^*T_\cX) \to T^1_\xi  \\
  \to& \Ext^1_{\cO_\cC}(\Omega_\cC(\fx_1+\cdots + \fx_n), \cO_\cC)\to H^1(\cC,f^*T_\cX)\to T^2_\xi \to 0
\end{aligned}
\end{equation}
where
\begin{itemize}
\item $\Ext^0_{\cO_\cC}(\Omega_\cC(\fx_1+\cdots+\fx_n),\cO_\cC)$ is the space of infinitesimal automorphisms of the domain $(\cC, \fx_1,\ldots, \fx_n)$, 
\item $\Ext^1_{\cO_\cC}(\Omega_\cC(\fx_1+\cdots + \fx_n), \cO_\cC)$ is the space of infinitesimal 
deformations of the domain $(\cC, \fx_1, \ldots, \fx_n)$, 
\item $H^0(\cC,f^*T_\cX)$ is the space of
infinitesimal deformations of the morphism $f$ for a fixed domain, and 
\item $H^1(\cC,f^*T_\cX)$ is
the space of obstructions to deforming the morphism $f$ for a fixed domain.
\end{itemize}
$T_\xi^1$ and $T_\xi^2$ are fibers of coherent sheaves $T^1$ and $T^2$ on the
moduli space $\MgXi$. The moduli space $\MgXi$ is equipped with a perfect obstruction theory
of virtual dimension $d^\vir_{g,\vi,\beta}$, where
\begin{equation}\label{eqn:vir-dim}
d^\vir_{g,\vi,\beta} =\int_\beta c_1(T_\cX) + (\dim \cX-3)(1-g) +n -\sum_{j=1}^n\age(\cX_{i_j}).
\end{equation}
Locally there is a two term complex $[E \to F]$ of locally free sheaves such that 
$$
\rank E -\rank F = d^\vir_{g,\vi,\beta}
$$
and $T^1$ and $T^2$ are the kernel and cokernel of $E\to F$, i.e.,
\begin{equation}\label{eqn:EF}
0 \to T^1 \to E \to F \to T^2 \to 0
\end{equation}
is an exact sequence of sheaves of $\cO_{\MgXi}$-modules. This determines a virtual fundamental class
(constructed in the algebraic setting in \cite{BeFa, LiTi1}): 
$$
[\MgXi]^\vir\in A_{d^\vir_{g,\vi,\beta}}(\MgXi;\bQ).
$$ 
Given a pair $(x,g)\in Ob(\cI \cX)$, where $x\in Ob(\cX)$ and $g\in \Aut_{\cX}(x)$, define
$r(x,g) = |\langle g\rangle|$. Then $(x,g)\mapsto |\langle g\rangle|$ defines a map
$r: \cI\cX \to \bZ_{>0}$ which is constant on each connected component $\cX_i$ of $\cI\cX$. Let
$r_i= r(\cX_i)$. The {\em weighted virtual fundamental class} is defined by 
$$
[\MgXi]^w := \Bigl(\prod_{j=1}^n r_{i_j}\Bigr) [\MgXi]^\vir.
$$

\subsection{Moduli of twisted stable maps to a formal smooth GKM stack} \label{sec:MghX}
Let $\hXY$ be the formal smooth GKM stack defined by an abstract stacky GKM graph $\vUp$, and
let $\hat{X}_{\vUp}$ be its coarse moduli space. Then 
$$
H_2(\hat{X}_{\vUp};\bZ) =\bigoplus_{\ep \in E(\Up)_c} \bZ [\ell_\ep].
$$
Let
$$
\Eff(\hXY)=\Big\{ \sum_{\ep\in E(\Up)_c} d_\ep [\ell_\ep]: d_\ep \in \bZ_{\geq 0} \Big\} \subset H_2(\hat{X}_{\vUp};\bZ)
$$
be the set of effective classes.  Given $g,n\in \bZ_{\geq 0}$ and $\hbeta\in \Eff(\hXY)$, let $\MghX$ be the moduli of genus $g$, $n$-pointed, 
degree $\hbeta$ twisted stable maps to $\hXY$, which is the analogue of $\MgX$ defined in Section \ref{sec:twisted}.
Let
$$
\cI \hXY =\bigsqcup_{i\in \IY} (\hXY)_i
$$
be disjoint union of connected components. Let $\MghXi$ be the analogue of $\MgXi$. Then we have a disjoint union
$$
\MghX =\bigsqcup_{\vi \in \IY^n} \MghXi.
$$
Each $\MghXi$ is equipped with a $T$-action together with a $T$-equivariant perfect obstruction theory of virtual dimension 
$d^\vir_{g,\vi,\hbeta}$, where 
$$
d^\vir_{g,\vi,\hbeta} =\int_{\hbeta} c_1(T_{\hXY}) + (\dim \hXY-3)(1-g) +n -\sum_{j=1}^n\age((\hXY)_{i_j}).
$$

Let $\cX$ be a smooth GKM stack, and let $\vUp$ be its stacky GKM graph. There is a $T$-equivariant morphism 
$j: \hXY \to \cX$, which induces a $T$-equivariant morphism 
$$
\cI \hXY=\bigsqcup_{ i \in I_\vUp }(\hat{\cX}_\vUp)_i.  \longrightarrow \cI\cX =  \bigsqcup_{i\in I} \cX_i. 
$$
and a surjective group homomorphism $j_*: H_2(\hat{X}_\vUp;\bZ)\to H_2(X;\bZ)$. Define $j:I_{\vUp}\to I$ such that
$j( (\hXY)_i) =\cX_{j(i)}$. We have $T$-equivariant morphisms
$$
\Mbar_{g,n}(\hXY,\hat{\beta}) \to \Mbar_{g,n}(\cX,j_*\hat{\beta}),\quad
\quad \Mbar_{g, (i_1,\ldots, i_n)}(\hat{\cX}_\vUp ,\hat{\beta}) 
\to \Mbar_{g, (j(i_1),\ldots, j(i_n))}(\cX,j_*\hat{\beta}).
$$

\subsection{Hurwitz-Hodge integrals}
\label{sec:hurwitz-hodge}

By Example \ref{ex:BG}, when $\cX=\cB G$ we have
$$
\cI \cB G =\bigsqcup_{c\in \Conj(G)} (\cB G)_c
$$
where $\Conj(G)$ is the set of conjugacy classes of $G$.
Give $\vc=(c_1,\ldots,c_n) \in \Conj(G)^n$, let 
$\MBGc = \Mbar_{g,\vc}(\cB G,\beta=0)$.
Then $\MBGc$ is a union of connected components
of $\MBG:=\Mbar_{g,n}(\cB G,0)$, and 
$$
\MBG=\bigsqcup_{\vc\in \Conj(G)^n}\MBGc. 
$$

We now fix a genus $g$ and $n$ conjugacy classes $\vc=(c_1,\ldots,c_n)\in \Conj(G)^n$.
Let $\pi: \cU\to \MBGc$ be the universal curve, and let $f: \cU\to \cB G$ be the universal
map. Let $\rho: G\to GL(V)$ be an irreducible representation of $G$, where $V$ is a finite
dimensional vector space over $\bC$. Then $\cE_\rho := [V/G]$ is
a vector bundle over $\cB G = [\pt/G]$. We have
$$
\pi_* f^* \cE_\rho =\begin{cases}
\cO_{\MBGc}, & \textup{if $\rho: G\to GL(1,\bC)$ is the trivial representation}, \\
0, & \textup{otherwise}.
\end{cases}
$$
The $\rho$-twisted Hurwitz-Hodge bundle $\bE_\rho$ can be defined as the 
dual of the vector bundle $R^1\pi_* f^* \cE_\rho$.   
If $\rho=1$ is the trivial representation, then
$\bE_1 = \ep^*\bE$, where $\ep:\Mbar_{g,\vc}(\cB G) \to \Mbar_{g,n}$, and
$\bE \to \Mbar_{g,n}$ is the Hodge bundle of $\Mbar_{g,n}$. So
$\rank \bE_1 =g$. If $\rho$ is a nontrivial irreducible representation, it follows
from the Riemann-Roch theorem for twisted curves (see Section \ref{sec:orbRR})  that
\begin{equation}
\rank \bE_\rho = \rank (\cE_\rho) (g-1) + \sum_{j=1}^n \age_{c_j}(\cE_\rho),
\end{equation}
where $\age_{c_j}(\cE_\rho)$ is given as follows.  Choose $g\in c_j$. Let $s>0$ be 
the order of $g$ in $G$,  let $N=\rank \cE_\rho =\dim V$.
If the eigenvalues of $\rho(g)\in GL(V)=GL(N,\bC)$ are $\{ e^{\frac{2\pi\sqrt{-1}}{s} l_i}: 1\leq i\leq N \}$,
where $l_i \in\{0,1,\ldots, s-1\}$,  then
$$
\age_{c_j}(\cE_\rho)=\frac{l_1+\cdots + l_N}{s}. 
$$
The definition is independent of choice of $g\in c_j$. Note that
$$
\sum_{j=1}^n\age_{c_j}(\cE_\rho) \in \bZ. 
$$
(If $G$ is {\em abelian} then any irreducible representation of $G$ is 1-dimensional,
so in this case $\rank(\cE_\rho)=1$ for any irreducible representation $\rho$ of $G$.)

\medskip

\begin{itemize}
\item {\em Hodge classes.} Given an irreducible representation $\rho$ of $G$, define
$$
\lambda_i^\rho = c_i(\bE_\rho) \in H^{2i}(\MBGc;\bQ), \quad i=1,\ldots, \rank\bE_\rho. 
$$

\item {\em Descendant classes}.
There is a forgetful map $\epsilon: \MBGc\to \Mbar_{g,n}$. Define 
$$
\bar{\psi}_j = \epsilon^*\psi_j \in H^2(\MBGc),\quad j=1,\ldots,n.
$$
\end{itemize}

{\em Hurwitz-Hodge integrals} are top intersection numbers of
Hodge classes $\lambda_i^\rho$  and  descendant classes $\bar{\psi}_j$: 
\begin{equation}
\int_{\MBGc}\bar{\psi}_1^{a_1}\cdots \bar{\psi}_n^{a_n}
(\lambda_1^{\rho_1})^{k_1}\cdots (\lambda_g^{\rho_g})^{k_g}.
\end{equation}
In \cite{Zh}, Jian Zhou describes an algorithm of computing Hurwitz-Hodge integrals, as
follows. By Tseng's orbifold quantum Riemann-Roch theorem \cite{Ts}, Hurwitz-Hodge integrals 
can be reconstructed from descendant integrals on $\MBGc$:
\begin{equation}
\int_{\MBGc}\bar{\psi}_1^{a_1}\cdots \bar{\psi}_n^{a_n}. 
\end{equation}
Jarvis-Kimura relate the descendant integrals on $\MBGc$ to those on $\Mbar_{g,n}$ \cite{JK}. 
We now state their result. Given $g\in \bZ_{\geq 0}$ and $\vc=(c_1,\ldots,c_n)\in \Conj(G)^n$, let
\begin{eqnarray*}
V^G_{g,\vc}&:=& \Big\{(a_1,b_1,\ldots, a_g, b_g, e_1,\ldots, e_n)\in G^{2g+n} \mid \\
&&  \quad \quad  a_1 b_1 a_1^{-1}b_1^{-1}\cdots a_g b_g a_g^{-1}b_g^{-1}= \prod_{j=1}^n e_j,\ \   e_j\in c_j\Big\}
\end{eqnarray*}
which is a finite set. The moduli of flat $G$-bundles over a compact Riemann surface of genus $g$ with markings
$c_1,\ldots, c_n$ is the quotient stack
$$
[V^G_{g,\vc}/G]
$$
where $G$ acts on $V^G_{g,\vc}$ by diagonal conjugation: 
$$
h\cdot (a_1,b_1,\ldots, a_g, b_g, e_1,\ldots, e_n) =
(h a_1 h^{-1}, h b_1 h^{-1},\ldots, h a_g h^{-1}, h b_g h^{-1}, h e_1 h^{-1}, \ldots, h e_n h^{-1}).
$$
\begin{example} If $h\in G$ then
$$
[V^G_{0,[h], [h^{-1}]}/G] \cong [\{ (h, h^{-1})\}/C_G(h)] = \cB C_G(h)
$$
where $C_G(h)$ is the centralizer of $h$ in $G$. 
\end{example}
If $G$ is {\em abelian} then each $c_i$ consists of a single element $h_i\in G$, and 
$$
V^G_{g,\vc} =\begin{cases}
G^{2g}\times \{h_1\}\times \cdots \{ h_n\} & \textup{if }h_1\cdots h_n=1,\\
\emptyset \quad \textup{(the emptyset)} & \textup{if }h_1\cdots h_n \neq 1;
\end{cases}
$$
so $[V^g_{g,\vc}/G] \cong G^{2g}\times \cB G$ if $h_1\cdots h_n=1$, and is empty otherwise.
In general, $\Mbar_{g,\vc}(\cB G)$ is non-empty if and only if $2g-2+n>0$ and $V^G_{g,\vc}$ is non-empty. 
In this case, the forgetful map $\ep:\Mbar_{g,\vc}(\cB G) \to \Mbar_{g,n}$ is of degree $|V^G_{g,\vc}|/|G|$.

\begin{theorem}[{Jarvis-Kimura \cite[Proposition 3.4]{JK}}]\label{thm:orb-psi}
Suppose that $2g-2+n>0$ and $V^G_{g,\vc}$ is nonempty. Then
$$
\int_{\Mbar_{g,\vc}(\cB G)}\bar{\psi}_1^{a_1} \cdots \bar{\psi}_n^{a_n} =\frac{|V^G_{g,\vc}|}{|G|}\int_{\Mbar_{g,n}} \psi_1^{a_1}\cdots \psi_n^{a_n}.
$$
\end{theorem}

\subsection{Orbifold GW invariants} \label{sec:GWinvariants}
There is a morphism $\epsilon:\MgXi \to \MgX$. Define $\bar{\psi}_i = \epsilon^* \psi_i$.

Suppose that the coarse moduli space $X$ is {\em projective}. Then
$\MgXi$ is proper. Let
$$
\gamma_j\in H^{d_j}(\cX_{i_j};\bQ) \subset H^{d_j+2\age(\cX_{i_j})}_\CR(\cX;\bQ),
$$
Define orbifold Gromov-Witten invariants
\begin{equation}\label{eqn:orbifoldGWprimary}
\langle \bar{\ep}_{a_1}\gamma_1,\ldots, \bar{\ep}_{a_n}\gamma_n\rangle_{g,\beta}^\cX
:=\int_{[\MgXi]^w} \prod_{j=1}^n \big( \ev_j^* \gamma_j  \cup \bar{\psi}_j^{a_j}\big) 
\end{equation}
which is zero unless
$$
\sum_{j=1}^n (d_j+ 2\age(\cX_{i_j})+2a_j) =2\left( \int_\beta c_1(T_\cX)+(1-g)(\dim \cX-3)+n\right).
$$

\subsection{Equivariant orbifold GW invariants}\label{sec:equivariant-GW}
Suppose that $\cX$ is equipped with a $T$-action, which induces a $T$-action on 
$\MgXi$ and on the perfect obstruction theory. Then there is a $T$-equivariant virtual fundamental class
$$
 [\MgXi]^{\vir,T} \in H_{2d^\vir_{g,\vi,\beta} }^T(\MgXi;\bQ).
$$

The weighted $T$-equivariant virtual fundamental class is defined by 
$$
[\MgXi]^{w,T} =\Bigl(\prod_{j=1}^n r_{i_j}\Bigr)[\MgXi]^{\vir,T}.
$$

Suppose  that $\MgXi$ is {\em proper}. (If the coarse moduli space $X$ is projective then $\MgXi$ is proper for
any $g,\vi,\beta$.) Given $\gamma_j^T\in H^{d_j}_T(\cX_{i_j};\bQ)\subset H^{d_j+2 \age(\cX_{i_j})}_{\CR,T}(\cX;\bQ)$ and $a_j\in \bZ_{\geq 0}$,
we define $T$-equivariant orbifold Gromov-Witten invariants
\begin{equation}\label{eqn:equivariant-orbGW}
\begin{aligned}
\langle\bar{\ep}_{a_1}(\gamma_1^T),\cdots,\bar{\ep}_{a_n}(\gamma_n^T)\rangle_{g,\beta}^{\cX_T}
:= & \int_{[\MgXi]^{w,T}} \prod_{j=1}^n \bigl(\ev_j^*\gamma_j^T \cup (\bar{\psi}_1^T)^{a_j}\bigr)\\
&\in \bQ[u_1,\ldots,u_m](\sum_{j=1}^n (d_j+ 2a_j) -2d^\vir_\vi).
\end{aligned}
\end{equation}
where $\bQ[u_1,\ldots,u_m](2k)$ is the space of degree $k$ homogeneous polynomials in 
$u_1,\ldots,u_l$ with rational coefficients, and 
$\bQ[u_1,\ldots,u_m](2k+1)=0$. In particular, 
$$
\langle\bar{\ep}_{a_1}(\gamma_1^T),\cdots,\bar{\ep}_{a_n}(\gamma_n^T)\rangle_{g,\beta}^{\cX_T}
=\begin{cases}
0, & \sum_{j=1}^n (d_j+2a_j) < 2d^\vir_{g,\vi,\beta},\\
\langle\bar{\ep}_{a_1}(\gamma_1),\cdots,\bar{\ep}_{a_n}(\gamma_n)\rangle_{g,\beta}^\cX\in \bQ, &
\sum_{j=1}^n (d_j+2a_j) = 2d^\vir_{g,\vi,\beta}.
\end{cases}
$$
where $\gamma_j \in H^{d_j}(\cX_{i_j};\bQ)$ is the image of $\gamma_j^T $ under the map
$H_T^{d_j}(\cX_{i_j};\bQ)\to H^{d_j}(\cX_{i_j};\bQ)$.

\subsection{Virtual localization}
Let $\cF=\MgXi^T\subset \MgXi$ be the substack of $T$ fixed points. 
The restriction of the exact sequence \eqref{eqn:EF} to $\cF$ splits into two exact sequences of
$\cO_{\cF}$-modules:
\begin{equation}\label{eqn:EFf}
0\to T^{1,f} \to E^f \to F^f \to T^{2,f}\to 0,
\end{equation}
\begin{equation}\label{eqn:EFm}
0\to T^{1,m} \to E^m \to F^m \to T^{2,m} \to 0,
\end{equation}
where \eqref{eqn:EFf} and \eqref{eqn:EFm} are the fixed and moving parts of \eqref{eqn:EF}, respectively. 
The 2-term complex  $[(F^f)^\vee \to (E^f)^\vee]$ defines a perfect obstruction theory on $\cF$; in other words, $\cF$ is 
equipped with a virtual tangent bundle
$$
T^\vir_\cF  = T^{1,f}- T^{2,f} = E^f- F^f.
$$
which might have different ranks on different connected components of $\cF$. This defines a virtual fundamental class
\cite{BeFa, LiTi1}
$$
[\cF]^\vir  \in A_*(\cF) 
$$
The virtual normal bundle of $\cF$ in $\MgXi$ is
$$
N^\vir = T^{1,m}-T^{2,m} = E^m - F^m.
$$ 
which might also have different ranks on different connected components of $\cF$, but
$$
\rank (T^\vir_\cF) + \rank (N^\vir)  = d_{g,\vi,\beta}
$$
is constant on $\cF$. 

By virtual localization \cite{Be2, GrPa},
\begin{equation}
\int_{[\MgXi]^{w,T}} \prod_{j=1}^n\bigl(\ev_j^*\gamma_j^T\cup (\bar{\psi}_j^T)^{a_j}\bigr)
= \int_{[\cF]^w}\frac{ i_T^*\left(\prod_{j=1}^n\big(\ev_j^*\gamma_j^T\cup (\bar{\psi}_1^T)^{a_j}\big)\right)}
{e^T(N^\vir)}
\end{equation}
where 
$$
[\cF]^w = \Big(\prod_{j=1}^n r_{i_j}\Big)[\cF]^\vir.
$$ 

Suppose that $\MgXi$ is not proper, but $\cF= \MgXi^T$ is proper.
(If $\cX$ is a smooth GKM stack then $\MgXi^T$ is proper for any $g,\vi, \beta$.)
We {\em define} 
\begin{equation}\label{eqn:residue-orb}
\langle\bar{\ep}_{a_1}(\gamma_1^T),\ldots,\bar{\ep}_{a_n}(\gamma_n^T)\rangle^{\cX_T}_{g,\beta}
= \int_{[\cF]^w}
\frac{i_T^* \left(\prod_{j=1}^n\big(\ev_j^*\gamma_j^T\cup (\bar{\psi}_j^T)^{a_j}\big)\right)}{e^T(N^\vir)}.
\end{equation}
When $\MgXi$ is not proper, the right hand side of \eqref{eqn:residue-orb} is a rational function
(instead of a polynomial) in $u_1,\ldots,u_m$. It can be nonzero when
$\sum_{j=1}^n (d_j+2a_j) < 2d^\vir_\vi$, and does not have a nonequivariant limit ($u_i\to 0$) 
in general.

\subsection{Formal equivariant orbifold GW invariants}
Let $\hXY$ be the formal smooth GKM stack defined by an abstract stacky GKM graph $\vUp$.  
Then there is a $T$-equivariant virtual fundamental class
$$
 [\MghXi]^{\vir,T} \in H_{2d^\vir_{g,\vi,\beta} }^T(\MghXi;\bQ).
$$
The weighted $T$-equivariant virtual fundamental class is defined by 
$$
[\MghXi]^{w,T} =\Bigl(\prod_{j=1}^n r_{i_j}\Bigr)[\MghXi]^{\vir,T}.
$$
Define
$$
[\MghX]^{w,T} =\sum_{\vi \in (\IY)^n} [\MghXi]^{w,T}.
$$

Let $\MghXi^T\subset \MghXi$ be the substack of $T$ fixed points. Then $\MghXi^T$ is a proper DM stack
equipped with a perfect obstruction theory which is the $T$ fixed the part of the restriction of the 
perfect obstruction theory on $\MghXi$, so we have
$$
[\MghXi^T]^\vir \in H_*(\MghXi^T)
$$
and
$$
[\MghXi^T]^w := \Bigl(\prod_{j=1}^n r_{i_j}\Bigr)[\MghXi^T]^\vir \in H_*(\MghXi^T). 
$$
Define 
$$
[\MghX^T]^w =\sum_{\vi\in (\IY)^n} [\MghXi^T]^w.
$$
Given  $\hga_1^T,\ldots, \hga_n^T \in \cH_{\vUp}$, we define 
\begin{equation} \label{eqn:formal-GW}
\langle\bar{\ep}_{a_1}(\hga_1^T),\ldots,\bar{\ep}_{a_n}(\hga_n^T)\rangle^{\vUp}_{g,\hbeta}
= \int_{[\MghX^T]^w}
\frac{i_T^*\left(\prod_{j=1}^n\big(\ev_j^*\hga_j^T\cup (\bar{\psi}_j^T)^{a_j}\big) \right)}{e^T(N^\vir)}.
\end{equation}

In the remainder of this subsection, we relate the above formal equivariant orbifold GW invariants to the equivariant
orbifold GW invariants defined in the previous subsection (Section \ref{sec:equivariant-GW}). 
Let $\cX$ be a smooth GKM stack and let $\vUp$ be its stacky GKM graph. We define
a surjective map $j_*:\Eff(\hXY) \longrightarrow\Eff(\cX)$ and an injective map 
$j^*: H_{\CR,T}(\cX;\mathcal{Q}_T) \to \cH_{\Up}$ as follows.
\begin{enumerate}
\item Let $I$, $I_\vUp$, $j: I_{\vUp}\to I$ be defined as in Section \ref{sec:MghX}. The surjective group homomorphism
$$
j_*: H_2(\hXY;\bZ)=\bigoplus_{\ep \in E_c(\Ga)} \bZ[\ell_\ep] \longrightarrow H_2(\cX;\bZ)
$$
restricts to a {\em surjective} map  
$$
j_*:\Eff(\hXY) \longrightarrow \Eff(\cX)
$$
where $\Eff(\cX)$ is the set of effective classes in $\cX$. Note that given $\beta\in \Eff(\cX)$, $\{ \hbeta\in \Eff(\hXY): j_*\hbeta =\beta\}$
is a finite set. 
\item There is a $\mathcal{Q}_T$-linear map
$$
j^*= \bigoplus_{\si\in V(\Up)} j_\si^*: H^*_{\CR,T}(\cX;\mathcal{Q}_T) \to \cH_{\vUp} 
= \bigoplus_{\si \in V(\Up)} H^*_{\CR,T}(\cX_\si;\mathcal{Q}_T)
$$
where $j_\si^*$ is induced by the inclusion $j_\si:\cX_\si \hookrightarrow \cX$.
\end{enumerate}
The following identity follows from the localization computations in Section \ref{sec:localization}.
\begin{proposition} Given nonnegative integers $g, a_1,\ldots, a_n$ an effective class $\beta \in \Eff(\cX)$, and
$\ga_1^T,\ldots, \ga_n^T\in H^*_{\CR,T}(\cX;\mathcal{Q}_T)$, we have 
$$
\langle\bar{\ep}_{a_1}(\gamma_1^T),\ldots,\bar{\ep}_{a_n}(\gamma_n^T)\rangle^\cX_{g,\beta}
=\sum_{\substack{ \hbeta  \in \Eff(\hXY) \\ j_*\hbeta =\beta} } 
\langle\bar{\ep}_{a_1}(j^*\gamma_1^T),\ldots,\bar{\ep}_{a_n}(j^*\gamma_n^T)\rangle^{\vUp}_{g,\hbeta}.
$$
\end{proposition}
Therefore, equivariant orbifold GW invariants of $\cX$ can be expressed in terms
of the more refined formal equivariant orbifold GW invariants of its stacky GKM graph $\vUp$.

\section{Torus fixed points and virtual localization}\label{sec:localization} 

\subsection{The fundamental group of a one-dimensional orbit}

The union of 1-dimensional $T$  orbits in $\cX$ is
$$
\cX^1\setminus \cX^T =\bigcup_{\ep\in E(\Up)} \fo_\ep
$$
where each 1-dimensional $T$ orbit $\fo_\ep$ is a $G_\ep$-gerbe over its coarse moduli $o_\ep \cong \bC^*$. 
Let $\fp_\ep \cong \cB G_\ep$ be a point in $\fo_\ep$ chosen as in Section \ref{GKM-graph}, and let
$$
H_\ep:=\pi_1(\fo_\ep,\fp_\ep)
$$
be the fundamental group of $\fo_\ep$. The projection $\fo_\ep\to o_\ep$ induces 
a surjective group homomorphism 
$$
H_\ep= \pi_1(\fo_\ep,\fp_\ep)\lra \pi_1(o_\ep, p_\ep) \cong \bZ
$$
whose kernel of is $G_\ep$. In other words, we have a short exact sequence of groups
\begin{equation}\label{eqn:foep}
1\to G_\ep \stackrel{j_\ep}{\lra}  H_\ep \stackrel{\phi_\ep}{\lra}  \bZ \to 1.
\end{equation}

Let $\ep\in E_c(\Up)$ be a compact edge, so that  $\ell_\ep\cong \bP^1$. Let $\si_x, \si_y \in V(\Up)$ be the two ends of the edge $\ep$, and let
$x=\fp_{\si_x}$ and $y=\fp_{\si_y}$ be the two torus fixed point corresponding to $\si_x$ and $\si_y$, respectively. 
Then $x =\cB G_{\si_x}$, $y=\cB G_{\si_y}$, and
$$
\cU_x:= \fl_\ep \setminus\{y\}\cong [\bC/G_{\si_x}],\quad \cU_y :=\fl_\ep \setminus \{x\} \cong [\bC/G_{\si_y}],\quad
\cU_x\cap \cU_y =\fo_\ep.
$$
The open embeddings $\fo_\ep\hra \cU_x$ and $\fo_\ep \hra \cU_y$ induce surjective group homomorphisms
\begin{equation}\label{eqn:foxy}
H_\ep=\pi_1(\fo_\ep) \stackrel{\pi_{(\ep,\si_x)} }{\lra} \pi_1(\cU_x)\cong G_{\si_x},\quad
H_\ep=\pi_1(\fo_\ep) \stackrel{ \pi_{(\ep,\si_y)} }{\lra} \pi_1(\cU_y)\cong G_{\si_y}.
\end{equation}
Recall that we also have the following two short exact sequences of groups:
\begin{equation}\label{eqn:Gsix}
1\to G_\ep \stackrel{j_{(\ep,\si_x)} }{\longrightarrow}  G_{\si_x}\stackrel{\phi_{(\ep,\si_x)}}{\longrightarrow} \mu_{r_{(\ep,\si_x)}}\to 1,
\end{equation}
\begin{equation}\label{eqn:Gsiy}
1\to G_\ep \stackrel{j_{(\ep,\si_y)} }{\longrightarrow}  G_{\si_y}\stackrel{\phi_{(\ep,\si_y)}}{\longrightarrow}  \mu_{r_{(\ep,\si_y)}}\to 1.
\end{equation}
Equations \eqref{eqn:foep}-\eqref{eqn:Gsiy}  fit into the following commutative diagram:
$$
\xymatrix{
1 \ar[r] & G_\ep \ar[r]^{j_{(\ep,\si_x)} } & G_{\si_x} \ar[r]^{\phi_{(\ep,\si_x)}} & \mu_{r_{(\ep,\si_x)}} \ar[r]  & 1 \\
1 \ar[r] & G_\ep   \ar[r]^{j_\ep}  \ar[u]_{\id_{G_\ep} } \ar[d]^{\id_{G_\ep} } & H_\ep \ar[r]^{\phi_\ep} \ar[u]_{\pi_{(\ep,\si_x)}} \ar[d]^{\pi_{(\ep,\si_y)}}& \bZ\ar[r]  \ar[u]\ar[d] & 1 \\
1 \ar[r] & G_\ep \ar[r]^{j_{(\ep,\si_y)} }  & G_{\si_y}  \ar[r]^{\phi_{(\ep,\si_y)}}&  \mu_{r_{(\ep,\si_y)}} \ar[r] &1 
} 
$$
where $\bZ\longrightarrow \mu_{r_{(\ep,\si_x)}}$ and  $\bZ\longrightarrow \mu_{r_{(\ep,\si_y)}}$ are given by 
$d\mapsto e^{ 2\pi\sqrt{-1} d/ r_{(\ep,\si_x)} }$ and  
$d\mapsto e^{ 2\pi\sqrt{-1} d/ r_{(\ep,\si_y)} }$, respectively.

\subsection{Torus fixed points in the moduli spaces} \label{sec:fixed}
Let $f:(\cC,\fx_1,\ldots, \fx_n)\to \cX$ be a twisted stable morphism which represents a $T$-fixed point in  $\Mbar_{g,n}(\cX,\beta)$.
Then there exists a surjective group homomorphism
$p: \tT \cong (\bC^*)^m  \longrightarrow T\cong (\bC^*)^m$ with finite kernel and a group homomorphism $\phi:\tT\longrightarrow \Aut(C,x_1,\ldots, x_n)$ such that
$p(t)\cdot f(z) = f(\phi(t)\cdot z)$ for all $z\in C$. The image of $f$ lies in the 1-skeleton $\cX^1$, the union of zero-dimensional and one-dimensional 
$T$ orbits in $\cX$.  In particular $f$ defines a twisted stable morphism with target $\hXY$ which represents a $T$-fixed point in 
$\Mbar_{g,n}(\hXY,\hbeta)$ where $\hbeta \in \Eff(\hXY)$ satisfies $j_* \hbeta=\beta$.

If $\cC_v$ is a connected component of $f^{-1}(\cX^T)$ then the image of $\cC_v$ is a $T$ fixed point $\fp_\si \cong \cB G_\si$ for some $\si\in V(\Up)$. 
If $O_e$ is a connected component of $f^{-1}(\cX^1\setminus \cX^T)$ then
$O_e\cong \bC^*$, and the image of $O_e$ is a 1-dimensional $T$ orbit $\fo_\ep$ for some $\ep \in E_c(\Up)$. The maps
$$
O_e \stackrel{f|_{O_e}} {\longrightarrow} \fo_\ep \to o_\ep
$$
induce
$$
\pi_1(O_e)= \bZ \stackrel{(f|_{O_e})_*}{\longrightarrow} \pi_1(\fo_\ep)=H_\ep \stackrel{\phi_\ep}{\longrightarrow} \pi_1(o_\ep)= \bZ.
$$
Let $\gamma_e\in H_\ep$ be the image of the generator of $\pi_1(O_e)= \bZ$ under $(f|_{O_e})_*$, and let $d_e = \phi_\ep(\gamma_e)\in \bZ$. Then $d_e>0$ is the degree of the map $O_e =\bC^*\longrightarrow o_\ep =\bC^*$.  

The map $f|_{O_e}:O_e\to \fo_{\ep}$ is of degree $d_e |G_\ep|$. We have
$$
\Aut(f|_{O_e}) \cong C_{H_\ep}(\gamma_e)/\langle \gamma_e\rangle. 
$$
In particular, if $G_\ep$ is trivial then $H_\ep= \bZ$ and $\Aut(f|_{O_e}) = \bZ/d_e\bZ$;
if $H_\ep$ is abelian then $\Aut(f|_{O_e}) = H_\ep/\langle \gamma_e\rangle$ and $|\Aut(f|_{O_e})| = d_e |G_\ep|$.

Let $\cC_e$ be the closure of $O_e$ in $\cC$. Then
$\cC_e$ is a football $\cF(r_u, r_v)$ and $f_e:=f|_{\cC_e}: \cC_e \to \fl_\ep$ is determined by 
$\gamma_e\in H_\ep$. Suppose
that $\si_x,\si_y\in V(\Up)$ are the two ends of the edge $\ep$. We define
$$
k_{(\ep,\si_x)} := \pi_{(\ep,\si_x)}(\gamma_\ep) \in G_{\si_x},\quad
k_{(\ep,\si_y)}:= \pi_{(\ep,\si_y)}(\gamma_\ep)\in G_{\si_y}.
$$
The map $f_e:\cC_e =\cF(r_u,r_v)\to \fl_\ep$ is representable, so $r_u$ and $r_v$ are the orders of
$k_{(\ep,\si_x)} \in G_{\si_x}$ and $k_{(\ep,\si_y)}\in G_{\si_y}$, respectively. In particular, the domain $\cC_e$ of $f_e$ is also determined by $\gamma_e$.
Let $\bar{f}_e: C_e =\bP^1 \to \ell_\ep=\bP^1$ be the map between coarse moduli spaces. Then 
$f_e([x,y])= [x^{d_e}, y^{d_e}]$ in terms of homogeneous coordinates on $\bP^1$.

\subsection{Torus fixed points and decorated graphs}\label{sec:graph-notation-orb}
Given a smooth GKM stack $\cX$, let
$$
\Mbar(\cX):= \bigsqcup_{\substack{g,n\in  \bZ_{\geq 0}\\ \beta\in \Eff(\cX)}}\Mbar_{g,n}(\cX,\beta)
$$
and let $\Mbar(\cX)^T$ be the $T$ fixed substack.

Given an abstract stacky GKM graph $\vGa$, let
$$
\Mbar(\hXY):=\bigsqcup_{\substack{g,n\in  \bZ_{\geq 0}\\ \hbeta\in \Eff(\hXY)}}\Mbar_{g,n}(\hXY,\hbeta)
$$
and let $\Mbar(\hXY)^T$ denote the $T$ fixed substack. By the discussion in Section \ref{sec:fixed},  if $\vGa$ is the stacky GKM graph 
of a smooth GKM stack $\cX$ then the morphism $\Mbar(\hXY) \to \Mbar(\cX)$ restricts to an isomorphism
$\Mbar(\hXY)^T \to \Mbar(\cX)^T$. In this subsection, we will describe $\Mbar(\hXY)^T$ for a general abstract stacky
GKM graph $\vGa$; in particular, this gives a description of $\Mbar(\cX)^T$ for any smooth GKM stack $\cX$. 

We fix an abstract stacky GKM graph $\vGa$, which defines a formal GKM stack $\hXY$. Let 
$\hXY^1=\bigcup_{\ep\in E(\Up)} \fl_\ep$ be the 1-skeleton of $\hXY$. Given a twisted stable map
$f:(\cC,\fx_1,\ldots,\fx_n)\to \hXY$ which represents a point in 
$\Mbar(\hXY)^T$, we define a decorated graph $\vGa=(\Ga, \vf, \vec{\gamma}, \vg, \vs, \vc)$ as follows.
\begin{enumerate}
\item (graph) $\Ga$ is a compact, connected 1 dimensional CW complex. We denote the set of vertices (resp. edges) in $\Ga$ 
by $V(\Ga)$ (resp. $E(\Ga)$).  For each connected component $\cC_v$ of $f^{-1}(\hXY^T) = f^{-1}(\{ \fp_\si:\si\in V(\Up) \})$, 
we associate a vertex $v\in V(\Ga)$. For each connected component $O_e \cong \bC^*$ of 
$f^{-1}(\hXY^1)\setminus f^{-1}(\hXY^T)$, we associate an edge $e\in E(\Gamma)$; the closure 
$\cC_e$ of $O_e$ in $\cC$ is a football. 
The set of flags of $\Gamma$ is defined to be
\begin{eqnarray*}
F(\Ga) &=&\{(e,v)\in E(\Ga)\times V(\Ga)\mid v\in e\} \\
&=& \{(e,v)\in E(\Gamma)\times V(\Ga)\mid \cC_v\cap \cC_e \textup{ is nonempty}\}.
\end{eqnarray*}

\item (label) For each vertex $v\in V(\Ga)$ let $C_v$ denote the coarse moduli of $\cC_v$.
Then $C_v$ is a curve (with at most nodal singularities) or a point, 
and $f(\cC_v)=\fp_{\si_v}$ for some $\si_v\in V(\Up)$. For each edge $e\in E(\Ga)$, $f(\cC_e)=\fl_{\ep_e}$ for some $\ep_e\in E(\Up)$. 
The {\em label map} $\vf: V(\Ga)\cup E(\Ga)\to V(\Up)\cup E(\Up)_c$
sends  a vertex $v\in V(\Ga)$ to  the vertex $\si_v \in V(\Up)$ and
an edge $e\in E(\Ga)$ to the edge edge $\ep_e \in E(\Up)_c$.
Moreover, $\vf$ defines a map from the graph $\Ga$
to the graph $\Up$: if $(e,v)\in F(\Ga)$ 
then $(\ep_e,\si_v)\in F(\Up)$. 

\item (degree) 
The {\em degree map} $\vec{\gamma}$ sends  an edge $e\in E(\Ga)$ to 
the conjugacy class $[\gamma_e] \in \Conj(H_{\ep_e})$, where $\gamma_e\in H_{\ep_e}$ 
is defined as in Section \ref{sec:fixed}. We call $[\gamma_e]$ 
the degree of the map $f_e=f|_{\cC_e}: \cC_e\to \fl_{\ep_e}$. The positive integer 
$d_e := \phi_\ep(\gamma_e)$  is the degree of the map $\bar{f}_e: C_e=\bP^1 \to \ell_{\ep_e}=\bP^1$ between
coarse moduli spaces. (Note that $\phi_\ep(\gamma_e)$ depends only on the conjugacy class
$[\gamma_e]$ of $\gamma_e$.)

\item (genus) The {\em genus map} $\vg:V(\Ga)\to \bZ_{\geq 0}$
sends a vertex $v\in V(\Ga)$ to a nonnegative integer $g_v$, where $g_v=0$ if $C_v$ is a point, and 
$g_v= h^1(C_v,\cO_{C_v})$ if $C_v$ is a curve.

\item (marking) The {\em marking map} $\vs: \{1,2,\ldots,n\}\to V(\Ga)$
sends $j$ to $v$ if $x_j \in C_v$. 

\item (monodromy) For any $v\in V(\Ga)$ we define $G_v = G_{\si_v}$.
Suppose that $j\in \{1,\ldots,n\}$ and $v\in \vs(j)\in V(\Ga)$.
Let $r_j$ be the cardinality of the inertia group
$\Aut(\fx_j)$ of the $j$-th marked point $\fx_j$, and let
$\xi_j$ be the generator of $\Aut(\fx_j)\cong \mu_{r_j}$ which acts on 
the tangent line $T_{\fx_j}\cC$ by $e^{2\pi\sqrt{-1}/r_j}$. 
Let $k_j \in G_v$ be the image of $\xi_j \in \Aut(\fx_j)$ under the group homomorphism 
$\Aut(\fx_j) \to \Aut(\fp_{\si_v}) = G_v$. The representability of $f$ implies
$\Aut(\fx_j)\to \Aut(\fp_{\si_v})$ is injective, so $r_j$ is equal to the order of $k_j \in G_v$.
The {\em monodromy map} $\vc$ sends a marking $j\in \{1,\ldots,n\}$ to the conjugacy class $c_j:=[k_j] \in 
\Conj(G_v)$ where $v=\vs(j)$,
\end{enumerate} 

The map $e \in E(\Gamma) \mapsto  [\gamma_e] \in \Conj(H_{\ep_e})$ determines
a map 
$$
(e,v) \in F(\Gamma) \mapsto  c_{(e,v)} := [\pi_{(\ep_e,\si_v)} (\gamma_e)] \in \Conj(G_{\si_v}).
$$
Given a flag $(e,v)$, let $\fy_{(e,v)}$ be the intersection point of $\cC_v$ and $\cC_e$.
(If $C_v$ is a point then $\cC_v=\{\fy_{(e,v)}\}$; if $C_v$ is a curve then $\fy_{(e,v)}$ is a node.) Let
$r_{(e,v)}$ be the cardinality of the inertia group of $\fy_{(e,v)}$, and let $\xi_{(e,v)}$ be the generator
of $\Aut(\fy_{(e,v)})\cong \mu_{r_{(e,v)}}$ which acts on the tangent line $T_{\fy_{(e,v)}} \cC_e$ by
$e^{2\pi\sqrt{-1}/r_{(e,v)}}$. Then the image of $\xi_{(e,v)}$ under the
injective group homomorphism $\Aut(\fy_{(e,v)})\to \Aut(\fp_{\si_v})=  G_v$ is
an element $k_{(e,v)}$ in the conjugacy class $c_{(e,v)}$. The representability of $f$ implies
$r_{(e,v)}$ is equal to the order of $k_{(e,v)}$. 

Given $v\in V(\Ga)$, we define $E_v\subset E(\Ga)$ and $S_v\subset \{1,\ldots,n\} $ by
\begin{align}\label{EvSv}
&
E_v=\{e\in E(\Ga): (e,v)\in F(\Ga)\}\notag\\
&
S_v=\{j\in \{1, \cdots, n\}:x_j\in C_v\}.
\end{align}
Given a conjugacy class $c\in \Conj(G_v)$, let $\bar{c}$ denote the conjugacy class
$\bar{c}=\{ k^{-1}: k\in c\}$. Define $\vc_v: E_v\cup S_v\to \Conj(G_v)$ by 
$\vc_v(e) = \overline{c_{(e,v)}}$ if $e\in E_v$, and
$\vc_v(j) = c_j$ if $j\in S_v$. Then $V_{g_v,\vc_v}^{G_v}$ is non-empty. 
Here we view $\vc_v$ as an element in $\Conj(G_v)^{E_v\cup S_v} = \Conj(G_v)^{n_v}$, where $n_v:= |E_v|+|S_v|$.

The inertia stack of $\hXY^T$ is
$$
\cI(\hXY^T) =\bigsqcup_{\si\in V(\Up)}\cI \fp_\si \cong \bigsqcup_{(\si,c)\in I^T_{\vUp}} (\cB G_\si)_c
$$
where 
$$
I^T_{\vUp}=\{ (\si,c):\si\in V(\Up), c \in \Conj(G_\si)\}.
$$
Connected components of  $\cI(\hXY^T)$ are in one-to-one correspondence with pairs $(\si,c) \in I^T_{\vUp}$.
The inclusion $\cI(\hXY^T) \hookrightarrow  \cI(\hXY)$ 
induces a surjective map $j_0: I^T_{\vUp} \to I_{\vUp}$ such that the image of $(\cB G_\si)_c$
under $j_0$ is contained in $(\hXY)_{j_0(\si,c)}$.  Let $G(\vUp)$ be the set of decorated graphs associated to some 
$T$ fixed twisted stable maps to $\hXY$. Then $G(\vUp)$ is a countable infinite set.
We have a map $\Mbar(\hXY)^T \to G(\vUp)$; let $\cF_\vGa \subset \Mbar(\hXY)^T$ be the preimage
of $\vGa \in G(\vUp)$. Then
$$
\Mbar(\hXY)^T =\bigsqcup_{\vGa \in G(\vUp)} \cF_\vGa.
$$ 
where each $\cF_\vGa$ is a union of connected components. 

\begin{definition}
Let $\vGa =(\Gamma,\vf, \vec{\gamma}, \vg,\vs,\vc)  \in G(\vUp)$. We define the genus of $\vGa$ to be
\begin{equation}\label{eqn:gGa}
g(\vGa):=  b_1(\Ga) +  \sum_{v\in V(\Ga)} g_v = |E(\Ga)|-|V(\Ga)|+1 + \sum_{v\in V(\Ga)} g_v 
\end{equation}
and define the degree of $\vGa$ to be
\begin{equation}\label{eqn:dGa}
\hbeta(\vGa):= \sum_{e\in E(\Ga)} d_e [\ell_{\ep_e}] 
=\sum_{\ep \in E_c(\Up)} \bigg(\sum_{e\in \vf^{-1}(\ep)} d_e \bigg)  [\ell_\ep]
\in \Eff(\hXY). 
\end{equation}
If the domain of the marking map $\vs$ is $\{1,\ldots,n\}$, we define $n(\vGa)=n$, and define
$$
\vec{i}(\vGa):= (j_0(\si_1,c_1),\ldots, j_0(\si_n,c_n))\in (I_{\vUp})^{n(\vGa)},
$$
where $\si_j = \vf\circ \vs(j) \in V(\Up)$ and $c_j \in \Conj(G_{\si_j})$. 
\end{definition}

Given nonnegative integers $g, n$ and an effective class $\hbeta\in \Eff(\hXY)$, define
$$
G_{g,n}(\vUp,\hbeta):=\{ \vGa\in G(\vUp): g(\vGa)=g, n(\vGa)=n, \hbeta(\vGa)=\hbeta \}. 
$$
Then $G_{g,n}(\vUp,\hbeta)$ is a finite set, and
\begin{equation}\label{eqn:FvGa}
\Mbar_{g,n}(\hXY,\hbeta)^T =\bigsqcup_{\vGa\in G_{g,n}(\vUp,\hbeta)} \cF_{\vGa}
\end{equation}
Given $\vec{i}= (i_1,\ldots, i_n) \in (I_{\vUp})^n$,  define
$$
G_{g,\vi}(\vUp, \hbeta) := \{ \vGa\in G(\vUp): g(\vGa)=g, \vi(\vGa)=\vi, \hbeta(\vGa)=\hbeta\}. 
$$
which is a subset of $G_{g,n}(\vUp,\hbeta)$. Then
$$
\Mbar_{g,\vec{i}}(\hXY,\hbeta)^T =\bigsqcup_{\vGa \in G_{g,\vec{i}}(\vUp,\hbeta)}\cF_{\vGa}.
$$

In the remainder of this section, we give an explicit description of $\cF_{\vGa}$ for 
each decorated graph $\vGa\in G(\vUp)$. We first introduce some notation. Let
\begin{eqnarray*}
V^S(\vGa) &=& \{ v\in V(\Ga): 2g_v-2+n_v>0\} =\{v\in V(\Ga): \cC_v \textup{ is a curve} \},\\
V^{0,1}(\vGa) &=& \{ v\in V(\Ga): g_v=0, |S_v|=0, |E_v|=1 \},\\
V^{1,1}(\vGa) &=& \{ v\in V(\Ga): g_v=0, |S_v|= |E_v|=1 \},\\
V^{0,2}(\vGa) &=& \{ v\in V(\Ga): g_v=0, |S_v|=0, |E_v|= 2\}. 
\end{eqnarray*}
Then $V(\Ga)$ is a disjoint union of $V^S(\vGa), V^{0,1}(\vGa), V^{1,1}(\vGa)$, and $V^{0,2}(\vGa)$. 
We say a vertex $v$ is stable if $v\in V^S(\vGa)$; otherwise we call it unstable. Let
$$
F^S(\vGa)=\{ (e,v)\in F(\Ga): v\in V^S(\vGa)\} 
$$
be the set of stable flags. 

Given an edge $e\in E(\Ga)$,  let $v,v'\in V(\Ga)$ be its two ends. 
By the discussion in Section \ref{sec:fixed}, the
map $f_e:=f|_{\cC_e}: \cC_e\to \fl_\ep$,
where $\ep=\vf(e)$, is determined by $\gamma_e \in H_{\ep_e}$. 
The automorphism group of $f_e$ is a finite group 
$$
\Aut(f_e) = \Aut(f|_{O_e}) \cong c_{H_\ep}(\gamma_e)/\langle \gamma_e\rangle.
$$
The moduli space of $f_e$ is
$$
\cM_e = \cB (\Aut(f_e)). 
$$

Given a stable vertex $v\in V^S(\vGa)$,  
the map $f_v:=f|_{\cC_v}: \cC_v\to \fp_\si =\cB G_v$, where $\si=\vf(v)$, represents a point
in $\Mbar_{g_v, \vc_v}(\cB G_v)$, where $\vc_v \in \Conj(G_v)^{E_v\cup S_v}$.  To obtain a $T$ fixed point $[f:(\cC,\fx_1,\ldots, \fx_n)\to \hXY]$, we glue
the the maps 
$$
\{ f_v:\cC_v \to \fp_{\si_v}\mid v\in V^S(\vGa)\}, \quad \{f_e:\cC_e\to \fl_{\ep_e} \mid e\in E(\Ga)\}
$$ 
along the nodes
$$
\{ \fy_{(e,v)}=\cC_e\cap \cC_v : (e,v)\in F^S(\vGa) \} \cup \{ \fy_v = \cC_v \mid v\in V^{0,2}(\vGa)\}.
$$
We define $\widetilde{\cM}_{\vGa}$ by the following 2-cartesian diagram
$$
\begin{CD}
\widetilde{\cM}_\vGa @>{f_E}>>  &  \displaystyle{ \prod_{e\in E(\Ga)} \cM_e }\\
@V{f_V}VV &  @V{\ev_E}VV \\
\displaystyle{ \cM_{\vGa}:=\prod_{v\in V^S(\vGa)} \Mbar_{g_v, \vc_v}(\cB G_v)  } @>{\ev_V}>> & 
\displaystyle{ \prod_{(e,v)\in F^S(\vGa)} \overline{\cI} \cB G_v \times \prod_{v\in V^{0,2}(\vGa)} \overline{\cI} \cB G_v }
\end{CD}
$$
where $\ev_V$ and $\ev_E$ are given by evaluation at nodes, 
and $\overline{\cI}\cB G_v$ is the rigidified inertia stack.
More precisely:
\begin{itemize}
\item  For every stable flag $(e,v)\in F^S(\vGa)$,  
let $\ev_{(e,v)}$ be the evaluation map at the node $\fy_{(e,v)}$,
\item For every $v\in V(\Ga)$,   let $\iota$ be the involution on $\cI\cB G_v$ induced by the involution $G_v\to G_v$ given by $h\mapsto h^{-1}$.
\item  Define
\begin{eqnarray*}
\ev_V &=& \prod_{ (e,v)\in F^S(\vGa)} \ev_{(e,v)}\, \\
\ev_E &=& \prod_{(e,v)\in F^S(\vGa)} \big(\iota\circ \ev_{(e,v)}\big) \times 
\prod_{\tiny \begin{array}{c} v\in V^{0,2}(\vGa)\\ E_v=\{e_1,e_2\} \end{array}} \ev_{(e_1,v)}\times \big(\iota\circ \ev_{(e_2,v)}\big) 
\end{eqnarray*}
\item If $v\in V^{0,2}(\vGa)$ and $E_v=\{e_1,e_2\}$, we define $r_v=r_{(e_1,v)} = r_{(e_2,v)}$, and define
$c_v = c_{(e_1,v)} = \overline{c_{(e_2,v)}}$.
\end{itemize}
We have
$$
\cF_\vGa = \widetilde{\cM}_\vGa/\Aut(\vGa)
$$
which is a proper smooth DM stack of dimension
$$
d_\vGa = \sum_{v\in V^S(\vGa)}(3g_v-3+n_v).
$$
It has a fundamental class
$$
[\cF_\vGa] =  c_\vGa [\cM_\vGa]\in  = A_{d_\vGa}(\cF_{\vGa};\bQ)= A_{d_\vGa}(\cM_{\vGa};\bQ) 
$$
where 
\begin{equation}
[\cM_\vGa] =\prod_{v\in V^S(\vGa)} \big[\Mbar_{g_v,\vc_v}(\cB G_v)\big],
\end{equation}
and 
\begin{equation}\label{eqn:cGa}
\begin{aligned}
c_\vGa & =  \frac{1}{|\Aut(\vGa)|\displaystyle{\prod_{e\in E(\Ga)}|\Aut(f_e)|}  } \cdot
\prod_{(e,v)\in F^S(\vGa)} \frac{|G_v|}{r_{(e,v)}|c_{(e,v)}|}\cdot \prod_{v\in V^{0,2}(\vGa)} \frac{|G_v|}{r_v|c_v|} \\
&= \frac{1}{|\Aut(\vGa)|\displaystyle{\prod_{e\in E(\Ga)}|\Aut(f_e)|}  } \cdot
\prod_{(e,v)\in F^S(\vGa)} \frac{|C_{G_v}(k_{(e,v)})|}{r_{(e,v)}}\cdot \prod_{v\in V^{0,2}(\vGa)} \frac{|C_{G_v}(k_v)|}{r_v}.
\end{aligned}
\end{equation}
In the second line in Equation \eqref{eqn:cGa} above,  $k_{(e,v)}$ (resp. $k_v$) is any element in the conjugacy 
class $c_{(e,v)}$ (resp. $c_v$), and
$C_{G_v}(k)$ denotes the centralizer of $k$ in $G_v$.

\subsection{Virtual tangent and normal bundles}
Given  $\vGa\in G(\vUp)$ and
a twisted stable map $f:(\cC,\fx_1,\ldots,\fx_n)\to \hXY$
which represents a point $\xi$ in $\cF_{\vGa}\subset \MghX$, the tangent space
$T^1_\xi$ and obstruction space $T^2_\xi$ of $\MghX$ at $\xi$ fits in
an following exact sequence of $T$-representations
\begin{equation}\label{eqn:B}
0\to B_1 \to B_2 \to T^1_\xi\to B_4 \to B_5\to T^2_\xi \to 0
\end{equation}
where
\begin{eqnarray*}
&& B_1 =  \Ext^0(\Omega_\cC(\fx_1+\cdots+\fx_n),\cO_\cC),\quad B_2 =   H^0(\cC,f^*T\cX)\\
&& B_4 = \Ext^1(\Omega_\cC(\fx_1+\cdots+ \fx_n),\cO_\cC),\quad B_5= H^1(\cC,f^*T\cX)
\end{eqnarray*}
Let $B_i^m$ and $B_i^f$ be the moving and fixed parts
of $B_i$, respectively; let $T^{i,m}_\xi$ and $T^{i,f}_\xi$ be the moving and fixed parts of $T^i_\xi$, respectively.
The exact sequence \eqref{eqn:B} splits into the following two exact sequences: 
\begin{equation}\label{eqn:Bfxi}
0\to B_1^f\to B_2^f\to T^{1,f}_\xi\to B_4^f\to B_5^f\to T^{2,f}_\xi\to 0,
\end{equation}
\begin{equation}\label{eqn:Bmxi}
0\to B_1^m\to B_2^m\to T^{1,m}_\xi \to B_4^m\to B_5^m\to T^{2,m}_\xi\to 0.
\end{equation}
Varying $\xi$ in the fixed locus $\cF_{\vGa}$ gives rise to the following two exact sequences of 
sheaves of $\cO_{\cF_{\vGa}}$-modules on $\cF_{\vGa}$:
\begin{equation}\label{eqn:Bf}
0\to B_1^f\to B_2^f\to T^{1,f}\to B_4^f\to B_5^f\to T^{2,f}\to 0,
\end{equation}
\begin{equation}\label{eqn:Bm}
0\to B_1^m\to B_2^m\to T^{1,m} \to B_4^m\to B_5^m\to T^{2,m}\to 0
\end{equation}
Here we abuse notation: $B_i^f$ (resp. $B_i^m$) are complex vector spaces in \eqref{eqn:Bfxi} (resp. \eqref{eqn:Bmxi}) and are 
sheaves over $\cF_{\vGa}$ in \eqref{eqn:Bf} (resp. \eqref{eqn:Bm}).
The restriction of the exact sequence \eqref{eqn:EF} to $\cF_{\vGa}$ also splits into two exact sequences of
$\cO_{\cF_{\vGa}}$-modules:
\begin{equation}
0\to T^{1,f} \to E^f \to F^f \to T^{2,f}\to 0.
\end{equation}
\begin{equation}
0\to T^{1,m} \to E^m \to F^m \to T^{2,m} \to 0.
\end{equation}

The dual complex of $[E^f \to F^f]$ is a perfect obstruction theory on the smooth proper DM stack 
$\cF_{\vGa}$; in other words, $\cF_{\vGa}$ is equipped with a virtual tangent bundle
$$
T^\vir_{\vGa}  = T^{1,f}- T^{2,f}
$$
As we will see in Section \ref{sec:domain}-\ref{sec:each-graph} below, $T^{1,f}=T\cF_{\vGa}$ is the tangent bundle
of the smooth DM stack $\cF_{\vGa}$, whereas $T^{2,f}=0$, so the virtual tangent
bundle is the tangent bundle. By \cite[Proposition 5.5]{BeFa},  
\begin{theorem}\label{thm:Fvir}
$$
[\cF_{\vGa}]^{\vir} = [\cF_{\vGa}] = c_{\vGa} \prod_{v\in V^S(\vGa)} [\Mbar_{g_v,\vc_v}(\cB G_v)]
$$
\end{theorem}

The virtual normal bundle of $\cF_{\vGa}$ in $\Mbar(\hXY)$ is
$$
N^\vir_{\vGa} = T^{1,m} - T^{2,m}.
$$
So 
\begin{equation} \label{eqn:euler}
\frac{1}{e_T(N^\vir_{\vGa})} =\frac{e_T(T^{2,m})}{e_T(T^{1,m})} =\frac{e_T(B_1^m)}{e_T(B_4^m)}
\frac{e_T(B_5^m)}{e_T(B_2^m)}. 
\end{equation}
We will compute $\displaystyle{ \frac{e_T(B_1^m)}{e_T(B_4^m)} }$ and
$\displaystyle{ \frac{e_T(B_5^m)}{e_T(B_2^m)} }$ in Section \ref{sec:domain} and Section \ref{sec:map}, respectively.

\subsection{Deformation of the domain} \label{sec:domain}

Recall that the nodes of $\cC$ are
$$
\{\fy_{(e,v)}=\cC_e\cap \cC_v: (e,v)\in F^S(\vGa)\} \cup \{\fy_v =\cC_v :(e,v)\in V^{0,2}(\vGa)\}.
$$
\subsubsection{Infinitesimal automorphisms of the domain} \label{sec:aut-orb}
\begin{eqnarray*}
B_1^f &=&\bigoplus_\edge \Hom(\Omega_{\cC_e}(\fy_{(e,v)}+\fy_{(e,v')}),\cO_{\cC_e})\\
&=& \bigoplus_\edge H^0(\cC_e, T\cC_e(-\fy_{(e,v)}-\fy_{(e,v')})\\
B_1^m&=& \bigoplus_{\substack{ v\in V^{0,1}(\vGa)\\ (e,v)\in F(\Ga)}} T_{ \fy_{(e,v)} }\cC_e 
\end{eqnarray*}
We define
$$
w_{(e,v)} :=  e^T(T_{\fy_{(e,v)} }\cC_e)=\frac{r_{(\ep_e,\si_v)}\bw_{(\ep_e,\si_v)}}{r_{(e,v)}d_e} \in 
H_T^2(\fy_{(e,v)})= M_\bQ.
$$

\subsubsection{Inifinitesimal deformations of the domain} \label{sec:deform-orb}
Given any $v\in V^S(\Ga)$, define
a divisor $\bx_v$ of $\cC_v$ by
$$
\bx_v=\sum_{i\in S_v} \fx_i + \sum_{e\in E_v} \fy_{(e,v)}.
$$
Then
\begin{eqnarray*}
B_4^f&=& \bigoplus_{v\in V^S(\vGa)} \Ext^1(\Omega_{\cC_v}(\bx_v),\cO_\cC) = \bigoplus_{v\in V^S(\vGa)}
T_{[(\cC_v,\bx_v)]} \Mbar_{g_v, \vi_v}(\cB G_v)\\
B_4^m &= & \bigoplus_{\substack{ v\in V^{0,2}(\vGa)\\ E_v=\{e,e'\} } }
T_{\fy_v}\cC_e\otimes T_{\fy_v} \cC_{e'} \oplus \bigoplus_{(e,v)\in F^S(\vGa)} 
T_{ \fy(e,v) }\cC_v\otimes T_{ \fy_{(e,v)} } \cC_e
\end{eqnarray*}
where $\displaystyle{ e^T (T_{\fy_v}\cC_e \otimes T_{\fy_v} \cC_{e'})= w_{(e,v)}+w_{(e',v)} }$ 
if $v\in V^{0,2}(\vGa)$ and $E_v=\{ e,e'\}$, and 
$\displaystyle{ e^T (T_{\fy_{(e,v)} }\cC_v \otimes T_{\fy_{(e,v)} } \cC_e) = w_{(e,v)}-\frac{\bar{\psi}_{(e,v)}}{r_{(e,v)}} }$
if $(e,v)\in F^S(\vGa)$.

\subsubsection{Unifying stable and unstable vertices}
From the discussion in Section \ref{sec:aut-orb} and Section \ref{sec:deform-orb},
\begin{equation} \label{eqn:Bonefour-orb}
\begin{aligned}
\frac{e^T(B_1^m)}{e^T(B_4^m)}=& 
\prod_{\substack{ v\in V^{0,1}(\vGa)\\ (e,v)\in F(\Ga)} } w_{(e,v)} 
\prod_{\substack{ v\in V^{0,2}(\vGa)\\ E_v=\{e,e'\} } }
\frac{1}{w_{(e,v)}+ w_{(e',v)} } \\
& \quad \cdot \prod_{v\in V^S(\vGa)}
\frac{1}{\prod_{e\in E_v}\Big(w_{(e,v)}-\frac{\bar{\psi}_{(e,v)}}{r_{(e,v)}}\Big) }.
\end{aligned}
\end{equation}

To unify the stable and unstable vertices, we use the following
convention for the empty sets $\Mbar_{0,(\{1\})}(\cB G)$ and $\Mbar_{0, ([h],[h^{-1}])}(\cB G)$,
where $1\in G$ is the identity element, and $[h]$ is the conjugacy class of $h\in G$.
Let $G$ be a finite group and let $w_1, w_2$ be formal variables.
\begin{itemize}
\item   $\Mbar_{0,(\{1\})}(\cB G)$ is a $-2$ dimensional space, and
\begin{equation}\label{eqn:one-orb}
\int_{\Mbar_{0,(\{1\})}(\cB G)}\frac{1}{w_1-\bar{\psi}_1}=\frac{w_1}{ |G| } 
\end{equation}
\item  $\Mbar_{0,([h],[h^{-1}])}(\cB G)$ is a $-1$ dimensional space, and
\begin{equation}\label{eqn:two-orb}
\int_{\Mbar_{0,([h], [h^{-1}])} (\cB G)}\frac{1}{(w_1-\bar{\psi}_1)(w_2-\bar{\psi}_2)}= \frac{1}{(w_1+w_2)\cdot |C_G(h)| }
\end{equation}
\begin{equation}\label{eqn:one-one-orb}
\int_{\Mbar_{0,([h],[h^{-1}])}(\cB G)}\frac{1}{w_1-\bar{\psi}_1} =\frac{1}{ |C_G(h)| }
\end{equation}
\end{itemize}
From \eqref{eqn:one-orb}, \eqref{eqn:two-orb}, \eqref{eqn:one-one-orb}, we obtain the following identities for non-stable vertices:
\begin{itemize}
\item[(i)] If $v\in V^{0,1}(\vGa)$ and $(e,v)\in F(\Ga)$, then $r_{(e,v)}=1$, and 
$$
|G_v| \int_{\Mbar_{0,(\{1\})} (\cB G_v) } \frac{1}{w_{(e,v)}-\bar{\psi}_{(e,v)}} = w_{(e,v)}.
$$
\item[(ii)] If $v\in V^{0,2}(\vGa)$, $E_v=\{ e,e'\}$, and $k_v \in c_{(e,v)}$, then $c_{(e',v)} =\overline{c_{(e,v)}} = [k_v^{-1}]$ and
\begin{eqnarray*}
&&  \frac{|C_{G_v}(k_v)|}{r_v}\cdot \frac{|C_{G_v}(k_v)|}{r_v} \cdot \int_{\Mbar_{0, ([k_v], [k_v^{-1}])}(\cB G_v)  }
\frac{1}{ \Big(w_{(e,v)} -\frac{\bar{\psi}_{(e,v)}}{r_v}\Big)\Big(w_{(e',v)}-\frac{\bar{\psi}_{(e',v)}}{r_v}\Big) }\\
&=&\frac{|C_{G_v}(k_v)|}{r_v} \cdot \frac{1}{w_{(e,v)} + w_{(e',v)}}.
\end{eqnarray*}
\item[(iii)]If $v\in V^{1,1}(\vGa)$ and $(e,v)\in F(\Ga)$, then
$$
\frac{|C_{G_v}(h)|}{r_{(e,v)}} \int_{\Mbar_{0,([h],[h^{-1}])}(\cB G_v)}\frac{1}{w_{(e,v)}- \frac{\bar{\psi}_1}{r_{(e,v)}} } =1.
$$
\end{itemize}
We then redefine $\cM_\vGa$ and $c_\vGa$ as follows:
\begin{equation}
\cM_\vGa =\prod_{v\in V(\Ga)} \Mbar_{g_v, \vi_v}(\cB G_v),\quad  [\cF_\vGa]=c_\vGa [\cM_\vGa],
\end{equation}
\begin{equation}\label{eqn:unified-c}
c_\vGa = \frac{1}{|\Aut(\vGa)|\prod_{e\in E(\Ga)} |\Aut(f_e)|}  \prod_{(e,v)\in F(\vGa)}\frac{|c_{G_v}(k_{(e,v)}|}{r_{(e,v)}}, 
\end{equation} 
where $k_{(e,v)}$ is an element in the conjugacy class $c_{(e,v)}$.

With the above conventions \eqref{eqn:one-orb}--\eqref{eqn:unified-c}, we may rewrite \eqref{eqn:Bonefour-orb} in 
the following form. 
\begin{proposition}\label{B1B4}
$$
\frac{e^T(B_1^m)}{e^T(B_4^m)}=
\prod_{v\in V(\Ga)} \frac{1}{ \prod_{e\in E_v}\Big(w_{(e,v)}-\frac{\bar{\psi}_{(e,v)}}{r_{(e,v)}}\Big)  }.
$$
\end{proposition}

The following lemma shows that the conventions \eqref{eqn:one-orb}, \eqref{eqn:two-orb}, and 
\eqref{eqn:one-one-orb} are consistent with
the stable case $\Mbar_{0,(c_1,\ldots, c_n)}(\cB G)$, $n\geq 3$.
\begin{lemma} Let $G$ be a finite group. Let $\vc=(c_1,\ldots, c_n)\in \text{Conj}(G)^n$. Let $w_1,\ldots,w_n$ be formal variables. Then 
\begin{enumerate}
\item[(a)]$\displaystyle{ \int_{\Mbar_{0,\vc}(\cB G)}\frac{1}{\prod_{i=1}^n(w_i-\bar{\psi}_i)}
=\frac{|V^{G}_{0,\vec{c}}|}{|G|\cdot  w_1\cdots w_n}\Big(\frac{1}{w_1}+\cdots \frac{1}{w_n}\Big)^{n-3} }$.
\item[(b)]$\displaystyle{\int_{\Mbar_{0,\vc}(\cB G)}\frac{1}{w_1-\bar{\psi}_1} =\frac{ |V^{G}_{0,\vec{c}}|}{|G|}w_1^{2-n} }$.
\end{enumerate}
\end{lemma}
\begin{proof}
The unstable cases $n=1$ and $n=2$ follow from the definitions
\eqref{eqn:one-orb} and \eqref{eqn:two-orb}, respectively.
The stable case ($n\geq 3$) follows from Theorem \ref{thm:orb-psi} and the well known identity below. 
$$
\int_{\Mbar_{0,n}}\psi_1^{a_1}\cdots \psi_n^{a_n}= 
\frac{(n-3)!}{a_1!\cdots a_n!}.
$$
\end{proof}

\subsection{Deformation of the map}\label{sec:map} 
We first introduce some notation. Given $\si\in V(\Up)$ and $c\in\Conj(G_\si)$, 
let $\bigl(T_{\fp_\si}\cX\bigr)^c$ denote the subspace of $T_{\fp_\si}\cX$ which is invariant
under the action of any $k\in c$ (or equivalently, of some $k\in c$). Then $\bigl(T_{\fp_\si}\cX\bigr)^c =\bigl(T_{\fp_\si}\cX\bigr)^{\bar{c}}$,
where $\bar{c}=\{k^{-1}: k\in c\}$. 

Consider the normalization sequence
\begin{equation}\label{eqn:normalize-orb}
\begin{aligned}
0 &\to \cO_\cC\to \bigoplus_{v\in V^S(\vGa)} \cO_{\cC_v} \oplus \bigoplus_{e\in E(\vGa)} \cO_{\cC_e}\\
& \to \bigoplus_{v\in V^{0,2}(\vGa)} \cO_{\fy_v}
\oplus \bigoplus_{(e,v)\in F^S(\vGa) } \cO_{ \fy_{(e,v)} }\to 0.
\end{aligned}
\end{equation}
We twist the above short exact sequence of sheaves
by $f^*T\cX$. The resulting short exact sequence gives
rise a long exact sequence of cohomology groups
\begin{eqnarray*}
0&\to& B_2 \to \bigoplus_{v\in V^S(\vGa)} H^0(\cC_v)\oplus
\bigoplus_{e\in E(\Ga)}H^0(\cC_e) \\
&\to& \bigoplus_{\tiny \begin{array}{c} v\in V^{0,2}(\vGa) \\ E_v=\{e,e'\} \end{array}} (T_{f(\fy_v)}\cX)^{c_{(e,v)}}
\oplus \bigoplus_{(e,v)\in F^S(\vGa)} \bigl(T_{f(\fy_{(e,v)})}\cX\bigr)^{c_{(e,v)}} \\ 
&\to& B_5 \to \bigoplus_{v\in V^S(\vGa)} H^1(\cC_v)\oplus
\bigoplus_{e\in E(\Ga)}H^1(\cC_e) \to 0.
\end{eqnarray*}
where $\displaystyle{H^i(\cC_v) = H^i(\cC_v, f_v^*T\cX)}$ and $\displaystyle{H^i(\cC_e) = H^i(\cC_e, f_e^*T\cX) }$
for $i=0,1$. 

\medskip

$f(\fy_v)=\fp_{\si_v}= f(\fy_{(e,v)})$. Given $(e,v)\in F(\Gamma)$, define
\begin{equation}\label{eqn:hev}
\bh(e,v) =e^T(\bigl(T_{\fp_\si} \cX\bigr)^{c_{(e,v)}}) 
=\prod_{\substack{ \ep \in E_{\si_v}\\ \phi_{(\ep,\si_v)}(c_{(e,v)})=1} } \bw_{(\ep,\si_v)}. 
\end{equation}
The map $B_1\to B_2$  sends
$H^0(\cC_e, T\cC_e(-\fy_{(e,v)}-\fy_{(e',v)}))$ isomorphically
to $H^0(\cC_e, f_e^*T\fl_{\ep_e})^f$, the 
fixed part of $H^0(\cC_e, f_e^*T\fl_{\ep_e})$.

It remains to compute
$$
\bh(v) := \frac{ e^T(H^1(\cC_v, f_v^*T\cX)^m) }{e^T(H^0(\cC_v, f_v^*T\cX)^m)} ,\quad
\bh(e) := \frac{e^T(H^1(\cC_e, f_e^*T\cX)^m)}{e^T(H^0(\cC_e, f_e^* T\cX)^m)}.
$$
The formulae of $\bh(v)$ and $\bh(e)$ will be given by Equation \eqref{eqn:hv} and
Equation \eqref{eqn:he} below. Before deriving these formulae ,
we introduce some notation. 
\begin{itemize}
\item If $v\in V^S(\vGa)$, then there is a cartesian diagram
$$
\begin{CD}
\tcC_v @>{\tf_v}>> \mathrm{pt}\\
@VVV @VVV\\
\cC_v @>{f_v }>> \cB G_v.
\end{CD} 
$$
Let $\hG_v$ denote the subgroup of $G_v$ generated by the monodromies
of the $G_v$-cover $\tcC_v\to \cC_v$. Then 
the number of connected components of 
$\tcC_v$ is $|G_v/\hG_v|$, and each connected component
is a $\hG_v$-cover of $\cC_v$.

\item Given a 1-dimensional representation $\phi: G_v\to \bC^*$, we $\phi$ of $G_v$, we define
$$
\Lambda_\phi^\vee(u)=\sum_{i=0}^{\rank \bE_\phi} (-1)^i \lambda_i^\phi u^{\rank \bE_\phi-i},
$$
where $\lambda_i^\phi \in A^i(\Mbar_{g_v, \vc_v}(\cB G_v))$ are Hurwitz-Hodge classes associated
to $\phi$, and $\rank\bE_\rho$ is the rank of the $\phi$-twisted Hurwitz-Hodge
bundle $\bE_\rho\to\Mbar_{g_v,\vc_v}(\cB G_v)$ (see Section \ref{sec:hurwitz-hodge}).  
\item Given a $G_v$ representation $V$, let $V^{G_v}$ denote the subspace
on which $G_v$ acts trivially. 
\end{itemize}
\begin{lemma}\label{lm:hv}
Suppose that $v\in V^S(\vGa)$ and $\vf(v)=\si\in V(\Up)$. Then
\begin{equation} \label{eqn:hv}
\bh(v) =  \frac{ \displaystyle{ \prod_{\ep \in E_\si } \Lambda^\vee_{\phi(\ep,\si)}(\bw_{(\ep,\si)})} }{ 
\displaystyle{ \prod_{\ep\in E_\si, \hG_v \subset G_\ep} \bw_{(\ep,\si)}} }
\end{equation}
\end{lemma}

\begin{proof} Let $\bC_\phi$ denote the 1-dimensional representation associated to $\phi: G_v\to \bC^*$. Then
$$
H^i(\cC_v, f_v^*T\cX) = \left(H^i(\tcC_v,\cO_{\tcC_v})\otimes T_\si \cX \right)^{G_v} 
\cong \bigoplus_{\ep \in E_\si} \left(H^i(\tcC_v, \cO_{\tcC_v})\otimes \bC_{\phi_{(\ep,\sigma)} } \right)^{G_v},
$$
where $H^0(\tcC_v,\cO_{\tcC_v})$ is the regular representation of $G_v/\hG_v$.
The surjective group homomorphism $G_v\to G_v/\hG_v$ induces an injective map
$\Hom(G_v/\hG_v,\bC^*) \to \Hom(G_v,\bC^*)$, whose image
contains $\phi^{-1}_{(\ep,\si)}$ iff $\hG_v\subset G_\ep$. So 
$$
e_T\Big( \big(H^0(\tcC_v,\cO_{\tcC_v}) \otimes \bC_{\phi_{(\ep,\sigma)} }\big)^{G_v}\Big)
=\begin{cases} \bw_{(\ep,\si)}, & \hG_v\subset G_\ep,\\
1, & \hG_v \not \subset G_\ep.
\end{cases}
$$
Therefore,
\begin{equation}\label{eqn:denominator}
e_T(H^0(\cC_v, f_v^*T\cX)^m)=e_T(H^0(\cC_v, f_v^*T\cX)) = \prod_{\ep\in E_\si, \hG_v \subset G_\ep} \bw_{(\ep,\si)}
\end{equation}
$$
\bigl(H^1(\tcC_v,\cO_{\cC_v})\otimes \bC_{\phi(\ep,\si)}\bigr)^{G_v} = \bE_{\phi(\ep,\si)}^\vee,
$$
so 
\begin{equation}\label{eqn:numerator}
e_T(H^1(\cC_v, f_v^*T\cX)^m)= e_T(H^1(\cC_v, f_v^*T\cX)) = \prod_{\ep \in E_\si } \Lambda^\vee_{\phi_{(\ep,\si)}} (\bw_{(\ep,\si)}).
\end{equation}
Equation \eqref{eqn:hv} follows from \eqref{eqn:denominator} and \eqref{eqn:numerator}. 
\end{proof}

Given any real number $x$, let $\lfloor x \rfloor$ denote the greatest integer which is less or equal to $x$, and let
$\langle x\rangle = x -\lfloor x\rfloor \in [0,1)$.

\begin{lemma} \label{lm:he}
Suppose that $e\in E(\Ga)$. Let $d=d_e\in \bZ_{>0}$, and let $\ep=\vf(e)\in E(\Up)_c$.
Define $\si, \si', \ep_i, \ep_i', a_i$ as in Section \ref{GKM-graph}.
Suppose that $(e,v), (e,v')\in F(\Ga)$, $\vf(v)= \si$, $\vf(v')=\si'$. 
Then any element in the conjugacy class $c_{(e,v)}\in \Conj(G_\si)$ acts on $T_{\fp_\si}\fl_\ep$ by multiplication
by $e^{2\pi\sqrt{-1}\langle d/r_{(\ep,\si)}\rangle }$, and acts on $T_{\fp_\si}\fl_{\ep_i}$ by 
$e^{2\pi\sqrt{-1} \alpha_i }$, where
$$
\langle \frac{d}{r_{(\ep,\si)}}\rangle , \alpha_1,\ldots, \alpha_{r-1} 
\in \Big\{ 0, \frac{1}{r_{(e,v)} },\ldots, \frac{r_{(e,v)}-1 }{r_{(e,v)}} \Big\}.  
$$
Define 
$$
\uu=r_{(\ep,\si)} \bw_{(\ep,\si)} = -r_{(\ep,\si')} \bw_{(\ep,\si')}.
$$
Then
\begin{equation}\label{eqn:he}
\bh(e)= \frac{ \displaystyle{ (\frac{d}{\uu})^{\lfloor\frac{d}{r_{(\ep,\si)} }\rfloor} } }
{\lfloor\frac{d}{r_{(\ep,\si)} }\rfloor !}
\frac{\displaystyle{  (-\frac{d}{\uu})^{\lfloor\frac{d}{r_{(\ep,\si')} }\rfloor } }  }
{\lfloor\frac{d}{r_{(\ep,\si')} }\rfloor !}\prod_{i=1}^{r-1} \mathbf{b}_i 
\end{equation}
where
\begin{equation}\label{eqn:b-orb}
\mathbf{b}_i=\begin{cases}
\displaystyle{\prod_{j=0}^{\lfloor da_i-\alpha_i \rfloor} (\bw_{(\ep_i,\si)} -(j+\alpha_i) \frac{\uu}{d})^{-1} }, 
& a_i\geq 0,\\
\displaystyle{\prod_{j=1}^{\lceil \alpha_i -da_i -1 \rceil } (\bw_{(\ep_i,\si)}+ (j-\alpha_i) \frac{\uu}{d}) },  
& a_i <0.
\end{cases}
\end{equation}
\end{lemma}

\begin{proof}
Let
$$
\bw_i =\bw_{(\ep_i,\si)},\quad i=1,\ldots,r-1.
$$
We have
$$
N_{\fl_\ep/\cX}=L_1\oplus \cdots \oplus L_{r-1}.
$$
\begin{itemize} 
\item The weights of $T$-actions on $(L_i)_{\fp_\si}$
and $(L_i)_{\fp_{\si'}}$ are 
$\bw_i$ and $\bw_i- a_i \uu $,
respectively. 
\item The weights of $T$-action on $T_{\fp_\si}\fl_\ep$
and $T_{\fp_{\si'}}\fl_\ep$ are 
$\displaystyle{ \frac{\uu}{r_{(\ep,\si)}} }$ and 
$\displaystyle{ \frac{-\uu}{r_{(\ep,\si')}} }$, respectively.
\item  Let $\fp_v =f_e^{-1}(\fp_\si), \fp_{v'}=f_e^{-1}(\fp_{\si'})$
be the two torus fixed points in $\cC_e$. Then
the weights of $T$-action on $T_{\fp_v}\cC_e$ and
$T_{\fp_{v'}}\cC_e$ are
$\displaystyle{ \frac{\uu}{d r_{(e,v)} } }$ and 
$\displaystyle{ \frac{-\uu}{d r_{(e,v')} } }$, respectively.
\end{itemize}
By \cite[Example 8.5]{Liu},  
$$
\ch_T(H^1(\cC_e, f_e^*L_i)-H^0(\cC_e, f_e^*L_i))
= \begin{cases}
- \displaystyle{ \sum_{j=0}^{ \lfloor da_i-\alpha_i\rfloor } e^{\bw_i - (j+\alpha_i) \frac{\uu}{d} } }, & a_i\geq 0,\\
\displaystyle{ \sum_{j=1}^{ \lceil \alpha_i -da_i-1\rceil } e^{\bw_i+(j-\alpha_i) \frac{\uu}{d} }  }, & a_i<0.
\end{cases}
$$
Note that $\bw_i - (j+\alpha_i) \uu$ and $\bw_i + (j-\alpha_i)\uu$  are nonzero for any $j\in \bZ$ since
$\bw_i$ and $\uu$ are linearly independent for $i=1,\ldots,r-1$. So
$$
\frac{e^T\left(H^1(\cC_e,f_e^*L_i)^m\right)}{e^T\left(H^0(\cC_e,f_e^*L_i)^m\right)}
=\frac{e^T\left(H^1(\cC_e,f_e^*L_i)\right)}{e^T\left(H^0(\cC_e,f_e^* L_i)\right)} =\mathbf{b}_i
$$
where $\mathbf{b}_i$ is defined by \eqref{eqn:b-orb}.

By \cite[Example 8.5]{Liu} again, 
\begin{eqnarray*}
&& \ch_T(H^1(\cC_e, f_e^*T\fl_\ep)-H^0(\cC_e, f_e^*T\fl_\ep))\\
&=& \sum_{\substack{j \in \bZ\\  -\langle \frac{d}{r_{(\ep,\si)}} \rangle \leq j \leq \frac{d}{r_{(\ep,\si)}}+ \frac{d}{r(\ep,\si')} - \langle \frac{d}{r_{(\ep,\si)}} \rangle }}
e^{\frac{\uu}{r_{(\ep,\si)}}- (j+ \langle \frac{d}{r_{(\ep,\si)}}\rangle)\frac{\uu}{d} } \\
&=&  = 1+\sum_{j=1}^{\lfloor \frac{d}{r_{(\ep,\si)}} \rfloor} e^{j\frac{\uu}{d}}
+ \sum_{j=1}^{\lfloor \frac{d}{r_{(\ep,\si')}}\rfloor} e^{-j\frac{\uu}{d}}.
\end{eqnarray*}
So
$$
\frac{e^T(H^1(\cC_e,f_e^*T\fl_\ep)^m)}{e^T(H^0(\cC_e,f_e^*T\fl_\ep)^m)}
= \prod_{j=1}^{\lfloor \frac{d}{r_{(\ep,\si)}}\rfloor}\frac{1}{j\frac{\uu}{d} }\prod_{j=1}^{\lfloor \frac{d}{r(\ep,\si')} \rfloor} \frac{1}{-j\frac{\uu}{d}}
= \frac{ \displaystyle{ \big(\frac{d}{\uu}\big)^{\lfloor\frac{d}{r_{(\ep,\si)}}\rfloor } } }{\lfloor\frac{d}{r_{(\ep,\si)}}\rfloor !}
\frac{ \displaystyle{ \big(-\frac{d}{\uu}\big)^{\lfloor\frac{d}{r(\ep,\si')}\rfloor} } }
{\lfloor\frac{d}{r(\ep,\si')}\rfloor !} 
$$
Therefore,
\begin{eqnarray*}
 \frac{e^T(H^1(\cC_e, f_e^* T\cX )^m )}
{e^T (H^0(\cC_e, f_e^*T\cX)^m) }
&=&\frac{e^T(H^1(\cC_e ,f_e^*T\fl_\ep)^m)}{e^T(H^0(\cC_e,f_e^*T\fl_\ep)^m)}
\cdot\prod_{i=1}^{r-1}\frac{e^T(H^1(\cC_e,f_e^*\cL_i)^m)}{e^T(H^0(\cC_e,f_e^*\cL_i)^m)}\\
&=&  \frac{\displaystyle{ \big(\frac{d}{\uu}\big)^{\lfloor\frac{d}{r_{(\ep,\si)}}\rfloor } } }
{\lfloor\frac{d}{r_{(\ep,\si)}}\rfloor !}
\frac{ \displaystyle{\big(-\frac{d}{\uu}\big)^{\lfloor\frac{d}{r_{(\ep,\si')} }\rfloor} } }
{\lfloor\frac{d}{r_{(\ep,\si')} }\rfloor !} \prod_{i=1}^{r-1} \mathbf{b}_i
\end{eqnarray*}
\end{proof}

From the above derivation, we conclude that
\begin{equation}\label{eqn:map}
\frac{e^T(B_5^m)}{e^T(B_2^m)} = \prod_{\substack{ v\in V^{0,2}(\vGa)\\ E_v= \{e,e'\} } } \bh(e,v) 
\cdot \prod_{(e,v)\in F^S(\vGa)} \bh(e,v) 
\cdot \prod_{v\in V^S(\vGa)}\bh(v)\cdot\prod_{e\in E(\vGa)}\bh(e)
\end{equation}
where $\bh(e,v)$, $\bh(v)$, and $\bh(e)$ are defined by \eqref{eqn:hev}, \eqref{eqn:hv}, \eqref{eqn:he}, respectively.
To unify the stable and unstable vertices, we define
$$
\bh(v):=\begin{cases}
\displaystyle{ \frac{1}{\bh(e,v)} }, & v\in V^{0,1}(\vGa)\cup V^{1,1}(\Ga), \quad E_v=\{e\},\\
\displaystyle{ \frac{1}{\bh(e,v)} =\frac{1}{\bh(e',v)} }, & v\in V^{0,2}(\vGa),\quad  E_v=\{e,e'\}.
\end{cases}
$$
In the above notation, \eqref{eqn:map} can be written as follows.
\begin{proposition}\label{B5B2}
$$
\frac{e^T(B_5^m)}{e^T(B_2^m)} 
= \prod_{v\in V(\Ga)} \bh(v)\cdot \prod_{(e,v)\in F(\Ga)} \bh(e,v) \cdot \prod_{e\in E(\Ga)}\bh(e).
$$
\end{proposition}

\subsection{Contribution from each graph} \label{sec:each-graph}
\subsubsection{Virtual tangent bundle} 
We have $B_1^f=B_2^f$, $B_5^f=0$. So
$$
T^{1,f}=B_4^f =\bigoplus_{v\in V^S(\Ga)}T_{[(\cC_v,\bx_v)]}\Mbar_{g_v, n_v},\quad
T^{2,f}=0.
$$
This completes the proof of Theorem \ref{thm:Fvir}.

\subsubsection{Virtual normal bundle} 
Equation \eqref{eqn:euler}, Proposition \ref{B1B4}, and Proposition \ref{B5B2} imply
\begin{equation}\label{eqn:euler-formula}
\frac{1}{e_T(N^\vir_\vGa)} = \prod_{v\in V(\Ga)}\frac{\bh(v)}{\prod_{e\in E_v}\Big(w_{(e,v)}-\frac{\bar{\psi}_{(e,v)}}{r_{(e,v)}}\Bigr)} 
\prod_{(e,v)\in F(\Ga)}\bh(e,v) \cdot \prod_{e\in E(\Ga)}\bh(e).
\end{equation}

\subsubsection{Integrand} 
Given $\si\in V(\Up)$, let $i_\si^*:\cH_{\Up} \to H^*_{\CR,T}(\fp_\si;Q_T)$ be the composition 
$$
\cH_{\Up}=\bigoplus_{\si\in V(\Up)} H^*_{\CR,T}(\cX_\si;Q_T) \to H^*_{CR,T}(\cX_\si;Q_T) \stackrel{i_\si^*}{\to} H^*_{\CR,T}(\fp_\si;Q_T)
$$
where the first arrow is projection to a direct summand, and 
the second arrow is induced by the inclusion $i_\si:\fp_\si \to \cX_\si$.  
Given $\vGa\in G_{g,\vi}(\Up,\hbeta)$, let
$$
i_\vGa^* :A^*_T(\MgXi)\to A^*_T(\cF_\vGa) \cong A^*_T(\cM_\vGa)
$$
be induced by the inclusion $i_\vGa:\cF_\vGa \to \Mbar_{g,\vi}(\hXY,\hbeta)$. Then
\begin{equation}\label{eqn:integrand}
\begin{aligned}
&i_\vGa^* \prod_{j=1}^n\left(\ev_j^*\gamma_j^T \cup (\bar{\psi}_j^T)^{a_j}\right) \\
=&\prod_{\tiny \begin{array}{c} v\in V^{1,1}(\vGa)\\ S_v=\{j\}, E_v=\{e\}\end{array}} i^*_{\si_v}\gamma_j^T (-w_{(e,v)})^{a_j} 
\cdot \prod_{v\in V^S(\Ga)}\Bigl(\prod_{j\in S_v} i^*_{\sigma_v}\gamma_j^T\prod_{e\in E_v}\bar{\psi}_{(e,v)}^{a_j}\Bigr)
\end{aligned}
\end{equation}
To unify the stable vertices in $V^S(\vGa)$ and the unstable vertices in $V^{1,1}(\vGa)$ , 
we use the following convention: for $a\in \bZ_{\geq 0}$ and $h\in G$, we define
\begin{equation}\label{eqn:one-one-a-orb}
\int_{\Mbar_{0,([h],[h^{-1}])}(\cB G)}\frac{\bar{\psi}_2^a}{w_1-\bar{\psi}_1}= \frac{(-w_1)^a}{|C_G(h)|}.
\end{equation}
In particular, \eqref{eqn:one-one-orb} is obtained by setting $a=0$. With the 
convention \eqref{eqn:one-one-a-orb}, we may rewrite \eqref{eqn:integrand} as
\begin{equation}
i_\vGa^*\prod_{j=1}^n\left(\ev_j^*\gamma_j^T \cup (\bar{\psi}_j^T)^{a_j}\right)=
\prod_{v\in V(\Ga)}\Bigl(\prod_{j\in S_v} i^*_{\sigma_v}\gamma_j^T\prod_{e\in E_v}\bar{\psi}_{(e,v)}^{a_j}\Bigr).
\end{equation}

The following lemma shows that the convention \eqref{eqn:one-one-a-orb} is consistent with
the stable case $\Mbar_{0,(c_1,\dots,c_n)}(\cB G)$, $n\geq 3$.
\begin{lemma}  Let $n,a$ be integers, $n\geq 2$, $a\geq 0$. Let $\vc=(c_1,\ldots, c_n)\in \Conj(G)^n$. Then
$$
\int_{\Mbar_{0,\vc}(\cB G)} \frac{\bar{\psi}_2^a}{w_1-\bar{\psi}_1}=
\begin{cases} \displaystyle{ \frac{\prod_{i=0}^{a-1}(n-3-i)}{a!}w_1^{a+2-n}\frac{|V^{G}_{0,\vec{c}}|}{|G|} }, & n=2 \textup{ or } 0\leq a\leq n-3.\\
0, & \textup{ otherwise.}
\end{cases} 
$$
\end{lemma}
\begin{proof} The case $n=2$ follows from \eqref{eqn:one-one-a-orb}.
For $n\geq 3$, 
\begin{eqnarray*}
&& \int_{\Mbar_{0,\vc}(\cB G)}\frac{\psi_2^a}{w_1-\bar{\psi}_1}
=\frac{1}{w_1} \int_{\Mbar_{0,\vc}(\cB G)}\frac{\bar{\psi}_2^a}{1-\frac{\bar{\psi}_1}{w_1}}
=w_1^{a+2-n} \int_{\Mbar_{0,\vc}(\cB G)} \bar{\psi}_1^{n-3-a}\bar{\psi}_2^a\\
&& = w_1^{a+2-n} |V^{G}_{0,\vc}|\cdot \frac{1}{|G|}\cdot \frac{(n-3)!}{(n-3-a)! a_!}
=\frac{\prod_{i=0}^{a-1}(n-3-i)}{a!} w_1^{a+2-n}\frac{|V^{G}_{0,\vec{c}}|}{|G|}.
\end{eqnarray*}
\end{proof}

\subsubsection{The integral} \label{sec:integral}
Define
\begin{equation} \label{eqn:I}
I_{\vGa} := \int_{[\cF_{\vGa}]^\vir} \frac{i_{\vGa}^*\prod_{j=1}^n (\ev_j^*\hga_j^T\cup (\bar{\psi}_j^T)^{a_i})}{e^T(N_\vGa^\vir)}.
\end{equation}
By Theorem \ref{thm:Fvir}, Equation \eqref{eqn:euler-formula},  and Equation \eqref{eqn:integrand}, 
\begin{equation}
\begin{aligned}
I_{\vGa} = & \ c_\vGa \prod_{e\in E(\Ga)}\bh(e)\prod_{(e,v)\in F(\Ga)} \bh(e,v) \\
&\  \cdot
\prod_{v\in V(\Ga)}\int_{[\Mbar_{g_v,{\vc}_v}(\cB G_v)]}
\frac{\bh(v)\cdot \prod_{j\in S_v}\big(i^*_{\si_v}\hga_j^T \cup \bar{\psi}^{a_j}_j\big) }
{\prod_{e\in E_v}\Big(w_{(e,v)}-\frac{\bar{\psi}_{(e,v)}}{r_{(e,v)}}\Big)}. 
\end{aligned}
\end{equation}
(Recall that  $c_\vGa\in \bQ$ is defined by Equation \eqref{eqn:unified-c}.) 

\subsection{Sum over graphs} Let
$$
i_T^* :H^*_T(\Mbar_{g,\vi}(\hXY,\hbeta))\to H^*_T(\Mbar_{g,\vi}(\hXY,\hbeta)^T) 
$$
be induced by the inclusion $i_T:\Mbar_{g,\vi}(\hXY,\hbeta)^T\to \Mbar_{g,\vi}(\hXY,\hbeta)$. 
Then 
$$
\int_{[\Mbar_{g,\vi}(\hXY,\hbeta)^T]^{\vir,T}} \frac{i_T^*\prod_{j=1}^n (\ev_j^*\hga_j^T\cup (\bar{\psi}_j^T)^{a_j})}{e^T(N^\vir)}
= \sum_{\vGa\in G_{g,n}(\vUp,\hbeta)} I_{\vGa},
$$
where the contribution  $I_{\vGa}$ from the decorated graph $\vGa$ is given in Section \ref{sec:integral} above. We obtain:
\begin{theorem}\label{main-orb}
\begin{equation}\label{eqn:sum-orb}
\begin{aligned}
\langle \bar{\ep}_{a_1}(\hga_1^T)\cdots \bar{\ep}_{a_n}(\hga_n^T)\rangle^{\Up}_{g,\hbeta}
=& \sum_{\vGa\in G_{g,\vi}(\cX,\beta)} c_\vGa
\prod_{e\in E(\Ga)} \bh(e) \prod_{(e,v)\in F(\Ga)} \bh(e,v) \\
&\cdot \prod_{v\in V(\Ga)}
\int_{[\Mbar_{g,\vi_v} (\cB G_v)]^w }
\frac{\bh(v)\prod_{j\in S_v}\big(i_{\si_v}^* \hga_j^T\bar{\psi}_j^{a_j}\big)}
{\prod_{e \in E_v} \Big(w_{(e,v)}-\frac{\bar{\psi}_{(e,v)}}{r_{(e,v)}} \Big)}
\end{aligned}
\end{equation}
where $\bh(e)$, $\bh(e,v)$, $\bh(v)$ are given by
\eqref{eqn:he}, \eqref{eqn:hev}, \eqref{eqn:hv}, respectively, and we have
the following convention for the $v\notin V^S(\Ga)$:
\begin{eqnarray*}
&&\int_{\Mbar_{0,(\{1\})}(\cB G_v)}\frac{1}{w_1-\bar{\psi}_1}= \frac{w_1}{|G|},\\
&& \int_{\Mbar_{0,([h],[h^{-1}])}(\cB G_v)}\frac{1}{(w_1-\bar{\psi}_1)(w_2-\bar{\psi}_2)}=\frac{1}{|C_{G_v}(h)|\cdot(w_1+w_2)},\\
&& \int_{\Mbar_{0,([h],[h^{-1}]}(\cB G_v)}\frac{\bar{\psi}_2^a}{w_1-\bar{\psi}_1}=\frac{(-w_1)^a}{|C_{G_v}(h)|},\quad a\in \bZ_{\geq 0}
\end{eqnarray*}
where $h\in G_v$, $[h]\in \Conj(G_v)$, and $C_{G_v}(h)=\{ a\in G_v: aha^{-1}=h\}$ is the centralizer of $h$ in $G_v$.
\end{theorem}

\end{document}